\documentclass[12pt,a4paper]{article}
\usepackage[applemac]{inputenc}
\usepackage{amsmath,amsthm,tabularx,gastex,hhline,rotating,comment,xspace}
\usepackage[shortlabels]{enumitem}
\usepackage{ wasysym }
\usepackage{amsfonts,amssymb}
\usepackage{MnSymbol}

\usepackage{graphicx,psfrag}
\usepackage{hyperref}
\usepackage[usenames]{color}

\usepackage{tikz}
\usetikzlibrary{decorations.markings}

\newcommand{\Sig}{\Sigma}
\newcommand{\Gam}{\Gamma}
\newcommand{\GS}{\Gam^*/S}

\newcommand{\refthm}[1]{Theorem~\ref{#1}\xspace}
\newcommand{\refcor}[1]{Corollary~\ref{#1}\xspace}
\newcommand{\refdef}[1]{Defintion~\ref{#1}\xspace}
\newcommand{\reflem}[1]{Lemma~\ref{#1}\xspace}
\newcommand{\refprop}[1]{Propostion~\ref{#1}\xspace}
\newcommand{\refrem}[1]{Remark~\ref{#1}\xspace}

\newcommand{\reffig}[1]{Figure~\ref{#1}\xspace}
\newcommand{\refsec}[1]{Section~\ref{#1}\xspace}

\newcommand\lds{,\ldots ,}

\newcommand{\sse}{\subseteq}
\newcommand{\es}{\emptyset}
\newcommand{\sm}{\setminus}

\newcommand{\IFF}{if and only if } 

\newcommand{\slnf}{shortlex normal form\xspace}

\newcommand{\set}[2]{\left\{\, \mathinner{#1}\vphantom{#2}\; \left|\; \vphantom{#1}\mathinner{#2} \right.\,\right\}}
\newcommand{\oneset}[1]{\left\{\, \mathinner{#1} \,\right\}}

\newcommand{\abs}[1]{\left|\mathinner{#1}\right|}

\newcommand{\wh}[1]{\widehat{#1}}

\newcommand\eps{\varepsilon}

\newcommand\ov[1]{\overline{#1}}

\newcommand\oi[1]{{#1}^{-1}}
\newcommand\os[1]{\widetilde{#1}}

\newcommand\id{\mathop{\mathrm{id}}}

\newcommand\ZZ{\mathbb{Z}}
\newcommand\NN{\mathbb{N}}

\newcommand\HNN{\mathrm{HNN}}

\renewcommand{\phi}{\varphi}

\newcommand{\del}{\delta}

\newcommand\RAS[2]{\overset{#1}{\underset{#2}{\Longrightarrow}}}

\newcommand\ra[1]{\overset{#1}{\longrightarrow}}

\newcommand\LAS[2]{\overset{#1}{\underset{#2}{\Longleftarrow}}}
\newcommand\DAS[2]{\overset{#1}{\underset{#2}{\Longleftrightarrow}}}
\newcommand\OUTS[5]{#1
 \overset{#2}{\underset{#3}{\Longleftarrow}} #4
 \overset{#2}{\underset{#3}{\Longrightarrow}} #5}
\newcommand\INS[5]{#1
  \overset{#2}{\underset{#3}{\Longrightarrow}} #4
  \overset{#2}{\underset{#3}{\Longleftarrow}} #5}
\newcommand\RA[1]{\underset{#1}{\Longrightarrow}}

\newcommand\OUT[4]{#1
 \underset{#2}{\Longleftarrow} #3
 \underset{#2}{\Longrightarrow} #4}

 \newcommand\CRAS[1]{\overset{\circledast}{\underset{#1}{\Longrightarrow}}}

\newcommand\CLAS[1]{\overset{\circledast}{\underset{#1}{\Longleftarrow}}}
\newcommand\CDAS[1]{\overset{\circledast}{\underset{#1}{\Longleftrightarrow}}}

\newlength{\rels}
\newcommand\RC[1]{{\underset{\hspace{-.15em}#1}{\hspace{\rels}\circ\hspace{-.45em}\Rightarrow\hspace{\rels}}}}

\newcommand\DC[1]{{\underset{\hspace{-.15em}#1}{\hspace{\rels}\Leftarrow\hspace{-.45em}\circ\hspace{-.45em}\Rightarrow\hspace{\rels}}}}

\newtheorem{theorem}{{\bf Theorem}}[section]
\newtheorem{corollary}[theorem]{{\bf Corollary}}
\newtheorem{definition}[theorem]{{\bf Definition}}
\newtheorem{example}[theorem]{{\bf Example}}
\newtheorem{lemma}[theorem]{{\bf Lemma}}

\newtheorem{proposition}[theorem]{{\bf Proposition}}
\newtheorem{remark}[theorem]{{\bf Remark}}

\newenvironment{am}{\noindent\color{blue} AM: }{}
\newenvironment{ajd}{\noindent\color{red} AJD }{}
\newenvironment{vd}{\noindent\color{magenta} VD }{}
\newcommand{\vdd}[1]{
\begin{vd} #1 \end{vd}}

\newcommand{\ad}[1]{ \begin{ajd} #1 \end{ajd}}

\newcommand{\be}{\begin{enumerate}}
\newcommand{\ee}{\end{enumerate}}

\def\CRXX{{%
   \setbox0\hbox{$\Longrightarrow$}%
   \rlap{\hbox to \wd0{\hss$\circlearrowright$\hss}}\box0
}}

\def\CLXX{{%
   \setbox0\hbox{$\Longleftarrow$}%
   \rlap{\hbox to \wd0{\hss$\circlearrowright$\hss}}\box0
}}

\def\CDXX{{%
   \setbox0\hbox{$\Longleftrightarrow$}%
   \rlap{\hbox to \wd0{\hss$\circlearrowright$\hss}}\box0
}}

\def\CDVD{{%
   \setbox0\hbox{$\overset{*}\Longleftrightarrow$}%
   \rlap{\hbox to \wd0{\hss$\circlearrowright$\hss}}\box0
}}

\newcommand\CCDAS[2]{\overset{*}{\underset{#2}{\CDXX}}}
\newcommand\CCRA[1]{\underset{#1}{\CRXX}}

\newcommand\CCDA[1]{\underset{#1}{\CDVD}}

\title{Cyclic rewriting and conjugacy problems} 
\author{Volker Diekert and Andrew J.~Duncan and Alexei Myasnikov\\ 
}

\date{\today}
\begin{document}
\maketitle

\begin{abstract}
Cyclic words are equivalence classes of cyclic permutations of ordinary
words. When a group is given by a rewriting relation, a rewriting
system on cyclic words is induced, which is used to construct algorithms
to find minimal length elements of conjugacy classes in the group. These
techniques are applied to the universal groups of Stallings pregroups and 
in particular to free products with amalgamation, HNN-extensions and 
virtually free groups, to yield simple and intuitive algorithms and proofs
of conjugacy criteria.
\end{abstract}

\tableofcontents

\section{Introduction}\label{sec:intro}
Rewriting systems are used in the theory of groups and monoids to specify
presentations together with conditions under which certain algorithmic 
problems may be solved. Typically, presentations given by convergent
rewriting systems are sought as these give rise to algorithms for the 
word and geodesics problems. Recently, less stringent  
conditions on rewriting systems which still 
allow the word problem and/or the geodesics problem to be decided, have also been 
investigated: for example geodesic or geodesically perfect systems\cite{GilHHR07,ddm10}. In contrast to the classical case,
geodesically perfect systems are confluent, but  not necessarily  convergent or finite, and are designed to seek geodesics in a group, rather than 
 normal forms of elements.   In any case, all these systems depend on  rewriting of  strings of letters, or words, from the 
free monoid on the generating set of a group or monoid.

In this paper
we consider applications of rewriting systems to the conjugacy problem 
in groups. 
To this end we apply rewriting to cyclic words rather than ordinary words. 
Cyclic words can be viewed as sets of all cyclic permutations of  standard words or, 
 equivalently, as graphs, which are directed labelled cycles. 
 This allows us to construct algorithms
for finding representatives of minimal length in the conjugacy classes of elements  in groups.   
 
 We describe analogues of 
Knuth-Bendix completion for rewriting systems on cyclic words and consider how
to realise these procedures in  particular situations. 
Our approach to the completion processes on cyclic words is rather different from the one developed by 
Chouraqui in \cite{Chouraqui11}, where cyclic rewriting systems
are used to construct algorithmic solutions to the conjugacy and 
transposition problems in 
monoids, under suitable conditions. One significant difference is that we introduce 
certain new rewriting rules, which are specific to the cyclic rewriting. These rules are absolutely essential but are not  ``induced'' by standard string rewriting.
Furthermore, in \cite{Chouraqui11}
the rewriting systems considered are all finite, whereas here we allow  
infinite systems. As in the case of geodesically perfect string rewrite systems we do not require our systems to be convergent.
  This allows us to construct confluent
cyclic rewriting systems which 
are particularly suitable for working with the conjugacy problem in groups. 

We apply these techniques to the conjugacy problem in universal groups
of Stallings pregroups and fundamental groups of graphs of groups.  As a warm-up  we give short intuitive proofs of the
conjugacy criteria of free products with amalgamation \cite{mks66}
and in HNN-extensions (Collin's Lemma) \cite{ls77}. Moreover we are 
able to describe a linear time algorithm for the conjugacy 
problem in finitely generated virtually free groups. 
(Epstein and Holt \cite{EpsteinH06} 
have constructed a linear time algorithm for the
conjugacy problem in arbitrary hyperbolic groups. However, for this
special case we give  
a very simple construction  based on the underlying finite rewriting system.) 
 
Canonical examples of pregroups
and their universal groups arise from free products with amalgamation,
HNN-extensions and, more generally,  fundamental groups of  graphs of groups. 
 The conjugacy problem
may behave  badly with respect to these constructions: 
for example in 
\cite{Lockhart1989} an HNN extension $G=\HNN(H, t;\;  t^{-1}at =b)$ is constructed,
 where the base group $H$ has 
solvable conjugacy
problem and the elements $a$ and $b$ of $H$ are infinite cyclic, but
$G$ has unsolvable conjugacy problem. Several authors have studied
conditions under which amalgams, HNN-extensions and graphs of groups do
have solvable conjugacy problem, see for example \cite{Horadam1983,Horadam1989,HoradamFarr1994,Lockhart1993} and the references
therein. Our results show that the obstruction to deciding
the conjugacy problem in such groups arises only from the determination
of conjugacy of elements of length one, with respect to the corresponding
pregroup.  Thus, if the conjugacy problem in the group is undecidable our systems do not provide a computable rewriting, but they do indicate where the difficulties are. This also gives a different view-point on the results of papers \cite{BMR3,BMR2,BMR4,FMR1} where efficient generic algorithms for the conjugacy problem in free products with amalgamation and HNN-extensions where constructed. These algorithms are fast correct partial algorithms that give the answer on most ("generic")  inputs, and do not give an answer only on a negligble set of inputs.

The structure of the paper is as follows. In Section \ref{sec:prelim} we
outline the transposition and conjugacy problems for monoids and groups
 and give a brief introduction to string rewriting systems. 
Section \ref{sec:cycwords} contains the definitions of cyclic words, cyclic
rewriting systems and the appropriate notions of 
geodesic and geodesically perfect cyclic systems, needed later in the paper.
In Section \ref{sec:compcyc} we consider how a semi-Thue system may
 be ``completed'' to give a larger, semi-Thue, system which is confluent on 
cyclic words. 
 This is possible under a weak termination condition, but the price is that,
in general, length increasing rules may be introduced. This leads in \ref{sec:CS}, \ref{sec:sct} and \ref{sec:cgp} to consideration of  
analogues of 
Knuth-Bendix completion processes in which we add context sensitive rules, 
 that 
rewrite transposed  
words directly to each other; when they  
 are of some globally bounded length. 
 
In Section \ref{sec:pregroup} we describe Stallings pregroups, their
universal groups and the rewriting systems to which they are naturally
associated. Section \ref{sec:cpug} contains the main results on conjugacy
in the universal groups of pregroups, namely \refthm{thm:mainpre}, \refcor{cor:hugo} and 
\refthm{thm:conjugation}. These results are applied to free products with
amalgamation, HNN-extensions and virtually free groups in Sections 
\ref{sec:cpag} and \ref{sec:cvfg}. 

\section{Preliminaries}\label{sec:prelim}
\subsection{Transposition, conjugacy and involution}\label{sec:sn}
Let $M$ be a monoid and $f,g\in M$.  Then $f$ and  $g$ are said to be  
{\em transpose}, if there exist elements $r,s\in M$ such that $f = rs$ and $g = sr$. We write $f \sim g$ to denote  transposition.
The elements $f$ and $g$ are called 
{\em conjugate}, if there exists an element $z\in M$ such that $fz = zg$.

In general these  definitions describe different relations. Indeed,
conjugacy is transitive, but not necessarily symmetric, while the transposition relation
is reflexive and symmetric, but not in general transitive. All transpose elements are conjugate. 
If the monoid $M$ is a group, then conjugacy is an equivalence relation and 
$f$ and $g$ are conjugate \IFF  there exists an element $z\in M$ such that $f= zg\oi z$.

Throughout $\Gam$ denotes an alphabet, which simply means it is a set, which 
might be finite or infinite in this paper. An element $a \in \Gam$ is called a \emph{letter} and an element $u$ in the 
free monoid $\Gam^*$ is called a \emph{word}. A non-empty word
 can be  written as 
$u = a_1\cdots a_n$, where $a_i\in \Gam$ and $n \geq 0$. The number 
$n$ is then called the \emph{length} of $u$ and  denoted  $\abs u$. 
The empty word has \emph{length} $0$ and  is denoted  $1$, as is customary for the 
neutral element in monoids or groups. 

A crucial, but elementary fact for free monoids is that transposition is 
equal to conjugacy. More precisely, in  free monoids $fz = zg$
implies that $f = rs$, $g = sr$, and $z = r(sr)^{m}$ for some $m \geq 0$. 
Essentially this implies a straightforward algorithm for  the conjugacy problem in free groups: 
on input elements $f$ and $g$ of  a free group first do cyclic reductions,
to cyclically reduce $f$ and $g$. 
This costs only linear time.  
Then check whether the cyclically reduced words $f$ and $g$ are transpose
by searching for the word $f$  as a factor of the word $g^2$. This is 
possible in linear time by a well-known pattern matching algorithm,
usually  attributed to  Knuth-Morris-Pratt
\cite{KMP}, although it was described earlier by Matiyasevich \cite{mat73}.

Frequently, sets and monoids come with an involution. An {\em involution} on a set $X$ is a permutation $a \mapsto \ov a$ such that $\ov{\ov a} = a$. An involution of a monoid satisfies in addition $\ov {xy}= \ov y\; \ov x$. If the monoid is a group $G$ then we always assume that the involution is given 
by the inverse, thus $\ov g = \oi g$ for group elements. If the alphabet 
$\Gam$ has an involution, then it is extended to $\Gam^*$ by defining 
$\ov{a_1\cdots a_n} = \ov{a_n} \, \cdots \, \ov{a_1}$ for $a_i\in \Gam$ and $n \geq 0$. {}From now on we always assume that $\Gam$ is  equipped with an involution $\ov{\phantom{u}}: \Gam \to \Gam$. Since the identity $\id_\Gam$ is an involution, this is no restriction. 
%


\subsection{Rewriting systems}\label{sec:rs}
Monoids and groups can be defined through a set of monoid generators 
$\Gam$ and a set of defining relations $S \sse \Gam^* \times \Gam^*$. A subset $S \sse \Gam^* \times \Gam^*$ is called  
a \emph{semi-Thue system}, or a \emph{string rewriting system}. 
Given $S$, we define a relation $\RAS{}S$, 
called a  \emph{ 
one-step rewriting relation}, on $\Gam^*$ by  
 $u \RAS{}S v$ \IFF 
$u = p\ell q$ and $v = pr q$ for some $(\ell, r) \in S$. 

Let $X$ be any set and 
$\RA{}\subseteq X \times X$ be a relation. 
 The iteration of at most $k$  steps of $\RA{}$ is denoted by $\RAS {\leq k} {}$ while 
 the reflexive and transitive closure of $\RA{}$  is denoted
by  $\RAS {*}{} $.
We also write $x \LAS{}{} y$ and  $x \LAS{*}{} y$ to denote  $y \RAS{}{} x$ 
and  $y \RAS{*}{} x$, respectively. 
The  reflexive, symmetric and transitive closure of $\RA{}$ is denoted by 
$\DAS {*} {}$. Elements $x \in X$ such that there is no $y$ with 
$x\RAS {}{} y$ are called \emph{irreducible}. 
 The relation $\RA{}$ is called:
\begin{enumerate}
\item \emph{strongly confluent}, if $\OUT y{}xz$ implies
$\INS y{\leq 1}{}wz$ for some $w$;
\item \emph{confluent}, if $\OUTS y*{}xz$ implies
$\INS y{*}{}wz$ for some $w$ and
\item \emph{Church-Rosser}, if $y\DAS * {}z$ implies
$\INS y{*}{}wz$ for some $w$.
\end{enumerate}

The following facts are well-known  and easy to prove, see e.g.{} 
 \cite{bo93springer,jan88eatcs}.

\begin{enumerate}
\item Strong confluence  implies confluence.
\item Confluence  is equivalent to Church-Rosser.
\end{enumerate}

A relation
$\RA{}\subseteq X \times X$ is called  \emph{terminating} (or \emph{Noetherian}), if there is no infinite chain
\begin{align*}
  x_0 \RA{} x_1 \RA{} \cdots \quad x_{i-1} \RA{} x_i  \RA{} \cdots
\end{align*}

For a semi-Thue system $S$ the equivalence relation 
 $\DAS {*} {S}$ is a congruence, hence the equivalence classes 
form a monoid which is denoted by $\Gam^*/S$. This is the quotient of the 
monoid $\Gam^*$ 
when $S$ is viewed as a set of defining relations. We also say that $S$ is confluent, terminating etc., whenever $\RA{S}$ has the corresponding property. 

The main interest in a terminating and confluent system $S$ stems from the fact that these properties  (together with some other natural condition on the computability of the one-step rewriting process) yield 
a procedure to solve the word problem in the quotient monoid $\Gam^*/S$. If $\Gam$ is finite, then decidability of the word problem is equivalent to 
the ability to compute shortlex normal forms:  first we endow the alphabet 
$\Gam$ with a linear order $\leq$. 
The \emph{shortlex normal form} for an element $g$ in a quotient monoid $\Gam^*/S$
is then the 
lexicographically first word among all geodesic words $u \in \Gam^*$
representing $g\in M$. Recall, that a  word $u \in \Gam^*$ is called a \emph{geodesic}, 
if $u$ has minimal length among all words representing the same element as $u$ 
in $\GS$.

\begin{example}\label{ex:fg}
If the involution $\ov{\phantom{u}}$ on $\Gam$ is without fixed points, then 
we can write $\Gam$ as a disjoint union $\Gam = \Sig \dot\cup \ov \Sig$. 
Then the rewriting system $S = \set {a \ov a \ra {} 1}{a \in \Gam}$ is strongly confluent and terminating; and the quotient monoid $\Gam^*/S$ defines the free group 
$F(\Sig)$. In this case geodesics are unique. 
\end{example}

\subsection{Thue systems}\label{sec:ts}
A semi-Thue system $S$ is called a \emph{Thue system}, if $S$ does not contain any 
length increasing rules and all length preserving rules are symmetric. 
This means $(\ell, r ) \in S$ implies $\abs \ell \geq \abs r$ and that  
$\abs \ell = \abs r$ implies $(r,\ell) \in S$, too. The set $S$ of a Thue system splits naturally into two parts $S = R \dot\cup T$, where 
$R$ contains the length reducing rules and $T$ contains the symmetric 
length preserving rules. In particular, $R\cap R^{-1} = \es$ and 
$T =T^{-1}$, where, as usual, $P^{-1} = \set{(y,x)} { (x,y) \in P}$ for any 
relation $P$. 

A Thue system  $S$ is called
\emph{geodesic}, if starting from any word $u$ and applying only length decreasing rules
we eventually obtain a {\em geodesic} word $v$ (a shortest word in the set  $\{v\mid u \DAS*S v\}$). Thus, we have $u \RAS*{R} v$ for 
some geodesic word $v$. 

A  
confluent, geodesic, Thue system is called 
\emph{geodesically perfect}.  
This means whenever 
$u \DAS*S v$, then we can first compute geodesics
$u \RAS*R \wh u$ and $v\RAS*R\wh v$, by applying length reducing rules, and then we can transform $\wh u$ into 
$\wh v$ by symmetric rules {}from $T$, that is  $\wh u \DAS*T \wh v$
(which in turn is equivalent to $\wh u \RAS*T \wh v$). Thus the following statements are equivalent for geodesically perfect systems. 

\begin{enumerate}
\item   $u \DAS*S v.$ 
\item $\exists\,  \wh u , \wh v: \; u \RAS*R \wh u \RAS*T \wh v  \LAS*R v.$
\end{enumerate}

\subsection{Cyclic words and cyclic rewriting}\label{sec:cycwords}

There are two principal ways of introducing cyclic words over an alphabet $\Gam$.  
The first one is based on combinatorics of words:  
in this case one defines a  \emph{cyclic word} as 
an equivalence class  of the transposition relation on $\Gam^*$. 
Thus, if $w \in \Gam^*$ then the cyclic word represented by $w$ is 
the set $w_{\sim} = \set{vu \in \Gam^*}{uv = w}$.  
The second one, defines the cyclic word represented by $w$ to be
the directed, $\Gam$-labelled, cycle graph $C_w$, 
such that the label of the cycle, when read with orientation, 
starting  at an appropriate vertex, is $w$. 
More precisely, if $w = a_1 \ldots a_n, n >0,$ then $C_w$ is a 
directed graph with vertices $v_1, \ldots, v_n$ and directed edges 
$e_1 = (v_1 \to v_2), \ldots, e_{n-1} = 
(v_{n-1} \to v_n), e_n = (v_n \to v_1)$ where each edge $e_i$ is  
labelled by $a_i$ 
In the graph-theoretic  version, an ordinary word $w \in \Gam^*$ 
can be viewed as a directed $\Gam$-labelled path-graph $P_w$: with 
vertices  $, v_1, \ldots, v_{n+1}$ and edges $e_1 
= (v_1 \to v_2), \ldots, e_{n} = (e_{n} \to e_{n+1})$  
with labels $a_1, \ldots,a_n$, respectively.  If $w $ is the empty word $1$ 
then $P_w$ and $C_w$ consist of  a single vertex.  
We regard the combinatorial and graph theoretic views of words and
cyclic words as different aspects of the same objects and pass
from one to the other without further comment.

{\em Graph rewriting} (or transformation)  is a well-established technique of 
computing with graphs. We refer to the book \cite{Rozenberg97} 
 for details.  In general,  a graph rewriting system consists of a set of 
graph rewriting rules of the form $(L,R)$, where $L$ and  $R$ are graphs.  
To apply such a rule to a given graph $G$ one finds a subgraph of $G$ 
isomorphic to $L$ and replaces it by $R$ according to some prescribed 
procedure.  

In our case the  graphs $G$ are cycles $C_w$, where   $w \in \Gam^*$ and the 
rewriting rules are of the following two types:

\begin{itemize}
\item [1)] $(P_\ell,P_r)$  for some $\ell, r \in \Gam^*$;
\item [2)] $(C_\ell,C_r)$ for some $\ell, r \in \Gam^*$, $\ell \neq 1$.
\end{itemize}

Application of a rule $(P_\ell,P_r)$ to a graph $C_w$ involves replacing some  path subgraph $P_\ell$ of $C_w$ by the path $P_r$. This can be clearly  visualised as in  
\reffig{fig:crew1}.  
\begin{figure}
\begin{center}
\begin{tikzpicture}[decoration={markings,mark=at position 0.5 with {\arrow{>}}}]

\def\cd{0.08}
\begin{scope}[xshift=-0.5cm]

	\draw[postaction={decorate}] (0,1) arc (90:135:1) node[above,left,anchor=south east] {$\ell$} arc (135:180:1);
	\draw[postaction={decorate}] (-1,0) arc (-180:90:1);
	\fill (0,1) circle (\cd );
	\fill (-1,0) circle (\cd );
	\end{scope}

	\node (mid) at (1.5,0) {$\RAS{\circ}{(\ell,r)}$};
\begin{scope}[xshift=0.5cm]

	\draw[postaction={decorate}] (3,1) arc (0:-45:1) node[below,right,anchor=north west] {$r$} arc (-45:-90:1);
	\draw[postaction={decorate}] (2,0) arc (-180:90:1);
	\fill (3,1) circle (\cd );
	\fill (2,0) circle (\cd );
		\end{scope}
\end{tikzpicture}
\caption{Cyclic rewriting when $(\ell,r) \in S$ and $\ell$ appears on the cycle.}
\label{fig:crew1}
\end{center}
\end{figure}
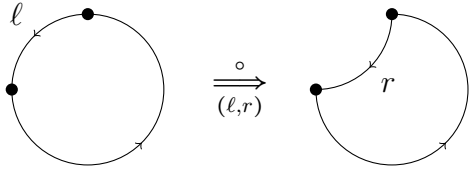
Application of the rule $(C_\ell,C_r)$ to $C_w$ is straightforward: if $C_w$ is isomorphic to $C_\ell$ (as a directed, labelled graph) then 
replace $C_w$ by $C_\ell$. Otherwise the rule does not apply. Clearly,  the
result of applying one of these rules  to cyclic word is  a cyclic word. 
A {\em rewriting system on cyclic words} is a set $T$ of rules of the type 1) and 2).  We write $C_u \RAS{}T C_v$ if $C_v$ can be obtained from $C_u$ by one of the rules from $T$. In this case we may also write 
$u_\sim \RAS{\circ}T v_\sim$ or 
$u \RAS{\circ}T v $. The definitions of Section \ref{sec:rs} 
apply to an arbitrary binary relation on a set $X$, and  
in particular to the relation $\RAS{}T$ on the set of cyclic words over  $\Gam^*$. Hence, we can talk about confluent, strongly confluent, terminating, etc. 
rewriting systems on cyclic words.

The subsystem  of $T$ consisting of the rules of type 1) corresponds
to a  string rewriting 
semi-Thue system  $S = \{ (\ell,r) \mid (P_\ell,P_r)\in T\}$. 
 On the other hand, let $S \sse \Gam^* \times \Gam^*$ be a semi-Thue system. 
Then $S$ composed on the right and left
with the relation $\sim$ defines  
a one-step relation $\RAS{\circ}S$ on  cyclic words. That is, we have
$u_\sim \RAS{\circ}S v_\sim$, \IFF there are words $u'$ and $v'$ 
such that $u\sim u'$, $u' \RAS{}S v'$, and $v'\sim v$. 
 Obviously, if a rule $(\ell, r) \in S$  is applied to 
$u_{\sim}$, then the  rewriting step $u_\sim \RAS{\circ}{(\ell,r)}  v_\sim$
may be understood  as applying   the rule $(P_\ell,P_r)$ to the graph $P_u$, as in  
\reffig{fig:crew1}. 

By analogy with string rewriting, we denote by   
  $\CRAS{S}$  the reflexive and transitive closure of $\RAS{\circ}{S}$;
write $u\LAS{\circ}{S}v$ and $u\CLAS{S}v$ for 
$v\RAS{\circ}{S}u$ and $v\CLAS{S}u$, respectively; and denote by
$\CDAS{S}$ the reflexive, symmetric, transitive closure of $\RAS{\circ}{S}$.

Neither confluence nor termination transfers from $S$ (defined on words) to 
$\RAS{\circ}S$ (defined on cyclic words). 
\begin{example}
\begin{enumerate}
\item
Let $\Gam=\{a,b,c,d\}$ and let $S$ consist of the following four rules
\[abc \ra{} bac,\quad cda  \ra{} dca,\quad  bad  \ra{} abd,
\quad dcb  \ra{} cbd.\]
To see that $\RAS{}S$ is confluent it is necessary to check all four  
cases where the left-hand sides of rules overlap. For example 
the left-hand side of $abc \ra{} bac$ overlaps with the left-hand side
of $cda  \ra{} dca$. Therefore we can rewrite $abcda$ in two ways: 
\[bacda\LAS{}S abcda\RAS{}S abdca.\] 
However 
\[bacda\RAS{}S badca \RAS{}S abdca,\]
so either way results in the same reduced word. The other three cases
are similar and so $\RAS{}S$ is confluent. However, the cyclic rewriting
system defined by $S$ is not confluent. In fact 
$abcd_\sim \RAS{\circ}S bacd_\sim$ and $abcd_\sim =bcda_\sim 
\RAS{\circ}S bdca_\sim$. Both $bacd_\sim$ and $bcda_\sim$ are irreducible
and they are not equal.
\item 
Let $\Gam=\{a,b\}$ and $S=\{ba \ra{} ab^2\}$. It is not difficult to
see that $\RAS{}S$ is terminating. However the relation $\RAS{\circ}S$ 
on cyclic words 
is non-terminating as 
\[ba_\sim  \RAS{\circ}S b^2a_\sim  \RAS{\circ}S 
b^3a_\sim  \RAS{\circ}S 
b^4 a_\sim  \RAS{\circ}S \cdots .\]
\end{enumerate} 
\end{example}

%
%
%
%

A semi-Thue system $S$ is called {\em C-confluent}, if $\RAS{\circ}S$ is confluent on cyclic words.  
If $W$ is subset of cyclic words, then we also say that $S$ is {\em C-confluent on $W$}, if $\RAS{\circ}S$ is confluent on all cyclic words in $W$.

In the rest of the section we consider some  general methods of transforming confluent semi-Thue systems into C-confluent  systems. 
\subsection{From confluence  to cyclic  confluence}\label{sec:compcyc}

Let $S \sse \Gam^* \times \Gam^*$ be a  {\em confluent}  
semi-Thue system   such that $G = \Gam^*/ S$ is a {\em group}. In this section we consider the general question (in the spirit of 
a Knuth-Bendix or Shirshov-Gr{\"o}bner completion) of how to enlarge the  system $S$ by adding new rules in order to obtain 
another system $\wh S$ such that the following hold:
\begin{enumerate}
\item\label{itone} $S \sse \wh S$ and $\Gam^*/ S = \Gam^*/ \wh S $ (i.e., $\wh S$ is a {\em conservative extension} of $S$);
\item\label{ittwo} $\RAS{\circ}{\wh S}$ is confluent on cyclic words (i.e., $\wh S$ is C-confluent). 
\end{enumerate}
Usually, we refer to $\wh S$ satisfying \ref{itone} as an {\em extension} of $S$ (omitting conservative). $\wh S$ satisfying \ref{itone} and \ref{ittwo} is termed a {\em C-extension} of $S$. Condition \ref{itone} ensures that $\wh S$ is still confluent (since $S$, and hence $\wh S$,  is Church-Rosser).

 Now we fix a  confluent semi-Thue system $S\sse \Gam^* \times \Gam^* $ such that $G = \Gam^*/ S$ is a group. For each letter  $a \in \Gam$ we can choose some fixed word
$\os a \in \Gam^*$  such that $a \os a = 1$ in $G$. We extend this definition 
(in a unique way) to all words of $\Gam^*$ as follows. Define $\os 1=1$ and assume that 
$\os u$ has been defined for all words $u$ of length at most $n$. 
Let $u=va$ be a word of length $n+1$, with $a\in \Gam$,
$v\in \Gam^*$. Then 
define $\os {u} = \os a \, \os v$.   Clearly, 
$\os{\os w} = w $ in $G$ for all words $w\in \Gam^*$.

For $x, y \in \Gam^*$ write $x > y$, if  $x\RAS{*}{S}pyq$ with $pq \neq 1$. Then $>$ is a partial order on $\Gam^*$. Since $x \RAS{*}{S} x = x1$ (here $1$ is the empty word) then $x > 1$ for every non-empty word $x$. We call the system $S$ \emph{weakly-terminating}  if the 
partial order $>$ is well-founded, i.e.,  there are no infinite chains
$x_1 > x_2 > x_3 > \cdots$. Clearly, if the system $S$ is terminating,
then it is weakly-terminating. Moreover, every semi-Thue system 
without length increasing rules is weakly-terminating. 
Note that the empty word $1$ is irreducible in every weakly-terminating system.  In particular, such a system does not have rules of the type $1 \to x\os x$ or $1 \to \os x x$, but $x\os x \RAS{*}{S} 1$ and $\os x x \RAS{*}{S} 1$ for any $x \in \Gam^*$, since $S$ is  Church-Rosser and $1$ is irreducible.

For a system $S$ define a semi-Thue system $\wh S$  by 
the following rules $u \RAS{}{\wh S}u'$ where: 
\begin{enumerate}
\item $u \RAS{}{S}u'$ (\emph{original rule}).
\item $u = qv$ and $u'= \os p r v$, if exists $pq\ra{}r \in S$, $p \neq 1 \neq q$ (\emph{prefix rule}).
\item $u = vp$ and $u'=   vr \os q$, if exists $pq\ra{}r \in S$, $p \neq 1 \neq q$ (\emph{suffix rule}).
\item $u'= \os p r\os q  $,  if exists $puq\ra{}r \in S$, $p \neq 1 \neq q$ (\emph{infix rule}).
\end{enumerate}

It is clear that $\wh S$ satisfies $S \sse \wh S$ and $\Gam^*/ S = \Gam^*/ \wh S $. 
 As before $\CRAS{\wh S}$ denotes  the reflexive and transitive closure of $\RAS{\circ}{\wh S}$.  

\begin{theorem}\label{thm:pptocc}
Let $S \sse \Gam^* \times \Gam^*$ be a confluent   weakly-terminating
semi-Thue system such that  $G = \Gam^*/ S$ is  a group. Then the following hold:
\begin{itemize}
\item [1)]  $u$ and $v$ are conjugate in $G$ \IFF 
$u\CRAS{\wh S}t \CLAS{\wh S}v$ for some (cyclic) word $t$. 
\item [2)] the rewrite system $\RAS{\circ}{\wh S}$ is confluent on 
(cyclic) words. 

\end{itemize}
 \end{theorem}

\begin{proof} To prove 1) observe first (by inspection of all the rules in $\wh S$) that 
\begin{equation}\label{eq:2}
u\CRAS{\wh S}t \CLAS{\wh S}v
\end{equation}
(in fact, even $u\CDAS{\wh S}v$) for some  word $t \in \Gam^*$  implies that $u$ and $v$ are conjugate in $G$. 

Assume now that   $u, v \in \Gam^*$ define conjugate elements  in $G$, i.e.,  $x u \os x \DAS{*}{S} v$ for some 
$x\in \Gam^*$.  We claim that in this case  there exists $t \in \Gam^*$ for which (\ref{eq:2}) holds. We proceed by Noetherian induction on $x$, i.e., by induction on the number of predecessors  of $x$ relative to $>$.

Since $S$ is Church-Rosser there exists $w \in \Gam^*$ such that  $xu\os x\RAS{*}{S}w \LAS{*}{S}v$. 
 If $x$ has no predecessors then $x = 1$ and the claim is obvious (in this case $t = w$). Thus, we may assume that the claim holds for all $y<x$.

In the reduction $xu\os x\RAS{*}{S}w $ the following  cases may occur.

{\em Case 1  (no overlap)}. 
Suppose one can factorise $w = x' u' x''$ in such a way that $x \RAS*S x'$, 
$u \RAS*S u'$, and $\os x \RAS*S x''$, then we are done since $x''x'= 1 $ in $G$, so $x''x' \RAS*S1$ ($1$ is $S$-irreducible), and hence: 
\[u \RAS*S u'\LAS*S u'x'' x' \sim x' u' x'' = w,\]
which proves the claim. Thus we may assume that there is no such factorisation.

{\em Case 2 (overlap)}. Assume now that $x \RAS*S yp$, 
$u \RAS*S qv$ such that $p\neq 1, q \neq 1$ and $pq \ra{} r$ is a rule of $S$. 
Then one has   $x u \os x = y(rv \os p) \os y = v$ in $G$ and $y < x$.  Hence, by induction, 
$rv \os p \CRAS{\wh S} t \CLAS{\wh S}v$  for some  word $t \in \Gam^*$. Notice, that 
we  can apply a  prefix rule to $qv$ and after a transposition  obtain
$u \CRAS {\wh S} rv \os p$. Therefore,  $u\CRAS{\wh S}t \CLAS{\wh S}v$ and the claim holds.

The argument for  the other possible overlap, when  $\os x  \RAS*S qy$ and 
$u \RAS*S vp$,  is similar and we omit it.  

{\em Case 3 (nesting)}. We are left to consider the following situation: 
$x \RAS*S yp$, 
$u \RAS*S s$, and $\os x  \RAS*S qz$ where $p\neq 1 \neq q$ and $psq \ra{} r$ is a rule of $S$. 
Again $x u \os x = y(r\,\os q\;\os p) \os y$ in $G$ and $y < x$.  Hence,  by induction, $r\,\os  q \; \os  p \CRAS{\wh S} t \CLAS{\wh S}v$, for some word $t$.
Applying an infix rule to $s$ and  
a transposition yields $ u \CRAS{\wh S} r\, \os q \;\os p$. The claim follows. 

This finishes the proof of 1). 
Statement 2) follows from 1) since, as  mentioned above,  $u\CDAS{\wh S}v$ implies that $u$ and $v$ are conjugate in $G$.
\end{proof}

Now we show that an extension $S^\circ$ (defined below) of $S$ 
which is, in this context, extremely natural is also a $C$-extension of $S$. 
Define $S^\circ$ to be the  extension of $S$ obtained by adding 
the rules $1\to a\os a$ and $1 \to \os a a$, for every $a \in \Gam$. Thus 
$$
S^\circ  = S \cup \{ 1\to a\os a, 1 \to \os a a \mid a \in \Gam\}
$$
and, 
since $G = \Gam^*/ S$ is a group, $S^\circ$ is indeed a $C$-conservative extension of $S$.  

\begin{theorem}\label{thm:pptocc2}
Let $S \sse \Gam^* \times \Gam^*$ be a confluent   weakly-terminating
semi-Thue system such that  $G = \Gam^*/ S$ is  a group. Then the following hold:
\begin{itemize}
\item [1)]  $u$ and $v$ are conjugate in $G$ \IFF 
$u\CRAS{S^\circ}t \CLAS{S^\circ}v$ for some (cyclic) word $t$. 
\item [2)] $S^\circ$ is a C-extension of $S$.  

\end{itemize}
 \end{theorem}
\begin{proof}
If $u\CRAS{S^\circ}t \CLAS{S^\circ}v$ for some  word $t$ then $u$ and $v$ are obviously conjugate in $G$.  Conversely, if $u$ and $v$ are conjugate in $G$, then by Theorem \ref{thm:pptocc} $u\CRAS{\wh S}t \CLAS{\wh S}v$  for some word $t$.  Observe, that application of a prefix, suffix, or infix rule from $\wh S$ is equivalent to a sequence of rule applications from $S^\circ$, so $u\CRAS{\wh S}t \CLAS{\wh S}v$  implies $u\CRAS{S^\circ}t \CLAS{S^\circ}v$. Now the result follows. 
\end{proof}

\subsection{From strong to cyclic confluence in groups}\label{sec:rcc}

The transformation of a semi-Thue system $S$ into the larger system $S^\circ$ 
described 
in \refsec{sec:compcyc} leads to length increasing rules. This is in some sense 
unavoidable. Indeed, assume we have $ab= c$ and $ba = d$ in the quotient 
$M= \Gam^*/S$ where $c,d\in \Gam$ are letters. In general we cannot expect that 
$c=d\in M$. But $c$ and $d$ are transpose, so we need cyclic rewriting rules to pass 
{}from $c$ to $d$ or vice versa. If we wish to do this by string rewriting and transpositions, then we are forced to pass {}from $c$ to $d$ via cyclic 
words of length at least 2. This is what happens in building $\wh S$ and
$S^\circ$.

Another idea is to introduce special 
rules which directly rewrite short cyclic words into each other, if 
they represent distinct 
conjugate elements.    In this case we have rules that rewrite cyclic words, 
but these rules are not induced by any string rewriting rules in the system (via equivalence relation $\sim$). We now make this precise.  

We start with a semi-Thue system $S \sse \Gam^* \times \Gam^*$, which 
we 
allow  
 to be infinite.  
Define 
\[m(S) = \sup \set {\abs \ell}{(\ell, r) \in S}.\] 
We say that $S$ is {\em left-bounded} if $m(S) < \infty$.  From now on we assume
 that the empty word is $S$-irreducible  and, to exclude trivial cases, that $2 \leq m(S)$.  To this end,  we say that $S$ is  a \emph{standard}  semi-Thue system, if it satisfies
the  two  conditions above: that is
\begin{enumerate}
\item $(1, r) \notin S$, for all non-empty words $r\in \Gam^*$, and 
\item $2 \leq m(S) < \infty$.
\end{enumerate}

A cyclic word $w_\sim$ is called $S$\emph{-short} if 
$\abs w \leq 2m(S) - 2$, and it is  called \emph{strictly} $S$\emph{-short}  if 
$\abs w < 2m(S) - 2$. When $S$ is fixed we refer to such words simply as short or strictly short.  

In the following let $C(S)$ denote any relation defined on the set of
 cyclic words 
which satisfies the following two conditions: 
\begin{enumerate}[C1]
\item\label{it:C1} If   $u_\sim \RAS{\circ}{S} v_\sim$, then $(u_\sim,v_\sim) \in C(S)$, i.e., 
$\RAS{\circ}{S} \subseteq C(S)$.
\item\label{it:C2} If $(u_\sim,v_\sim) \in C(S)$, then $u$ and $v$ are conjugate in $\Gam^*/S$.
\end{enumerate}

Later, we discuss the possibility of  constructing relations $C(S)$
with these properties. 
We write $u \RAS{\circ}{C(S)} v$ and
$v \LAS{\circ}{C(S)} u$ if   $(u_\sim,v_\sim) \in C(S)$ or if $u_\sim = v_\sim$. 
(Thus, both $\RAS{\circ}{C(S)}$ and $\LAS{\circ}{C(S)}$ are reflexive.)
Moreover, we use  $\CRAS{C(S)}$ and $\CDAS{C(S)}$ again, for the transitive, 
 and 
for the  symmetric and  transitive closure, respectively, of $\RAS{\circ}{C(S)}$.  
As $C(S)$ is a relation on cyclic words, when we say $C(S)$ is confluent,
or strongly confluent, unless we explicitly specify an alternative, we
mean confluent or strongly confluent on the set of all cyclic words.


\begin{theorem}\label{thm:CKB}
Let $S \sse \Gam^* \times \Gam^*$ be a standard strongly confluent semi-Thue system such that 
 $C(S)$ satisfies the two conditions \ref{it:C1} and \ref{it:C2} above. 
Then the following assertions are equivalent:
\begin{enumerate}[1.)]
\item The system $C(S)$ is confluent.
\item The system $C(S)$ is confluent on all short cyclic words $w$.
(That is if $w$ is short and $u \CLAS{C(S)} w \CRAS{C(S)} v$ then
there exists $t$ such that  $u \CRAS{C(S)} t \CLAS{C(S)} v$.)
\end{enumerate}
\end{theorem}

\begin{proof}
 We have to show only that if $C(S)$ is confluent on all short cyclic words, 
 then $C(S)$ is confluent. 
 
 First consider $u \LAS{\circ}{C(S)} w \RAS{\circ}{C(S)}v $ where 
 $\abs w \geq  2m(S) -1$. Then the two rules applied to the cyclic word $w_\sim$ are inherited from the semi-Thue system $S$. Since $w$ is long enough the corresponding left-hand sides overlap in the cyclic word $w$ at most once. Since $S$ is strongly confluent, 
 we see that  there is some cyclic word $t$ such that 
 \[u \RAS{\circ}{C(S)} t \LAS{\circ}{C(S)}v.\] Next, consider 
 \[u = w_k\LAS{\circ}{C(S)} \cdots \LAS{\circ}{C(S)}w_0 \RAS{\circ}{C(S)}v_1 \RAS{\circ}{C(S)} \cdots \RAS{\circ}{C(S)}v_m = v.\]
 We may assume that $m \geq k \geq 1$. We perform an induction on 
 $(k,m)$ in the lexicographical order. 
 
 If none of  $w_0 \lds w_{k-1}$ is short, then 
 by strong confluence of $S$ we have the following situation. 
 \[u \RAS{\circ}{C(S)}  w_k'\LAS{\circ}{C(S)} \cdots \LAS{\circ}{C(S)}w_1' \LAS{\circ}{C(S)}v_1 \CRAS{C(S)}v.\] 
 Thus, we are done by induction on $m$. 
 Therefore let  $w_{\ell}$ be a short cyclic word where $\ell \leq k-1$. By induction 
 on $k$ we see that there exists 
 \[w_\ell \CRAS{C(S)} t \CLAS{C(S)}v.\] 
 Moreover, $u \CLAS{C(S)} w_\ell$ and $C(S)$ is confluent on $w_\ell$ 
 because $w_\ell$ is a short cyclic word. Hence we find
 \[u \CRAS{C(S)} t' \CLAS{C(S)} t \CLAS{C(S)}v.\]
\end{proof}

\begin{corollary}\label{cor:csconj}
Let $S \sse \Gam^* \times \Gam^*$ be a standard strongly confluent semi-Thue system such that $\GS$ is a group and 
such that first, $C(S)$ is confluent 
on all short cyclic words and second,  it satisfies the two conditions \ref{it:C1} and \ref{it:C2} 
above. 
Then two words $u$ and $v$ are conjugate in $\GS$ \IFF 
there exists a cyclic word $t$ such that 
 \[u_\sim \CRAS{C(S)} t \CLAS{C(S)}v_\sim.\]
\end{corollary}

\begin{proof}
 Clearly,  $u_\sim \CDAS{C(S)} v_\sim$ implies conjugacy. 
 Now, if $u$ and $v$ are conjugate, then there is some $x$ 
 such that $xu \oi x \DAS{*}S v$. This implies 
 $xu \oi x_\sim \CDAS{C(S)} v_\sim$. We have $x\oi x \RAS*S 1$,  because 
 $S$ is standard and confluent, hence $ xu \oi x_\sim\CRAS{C(S)} u_\sim$.  We conclude $u_\sim \CDAS{C(S)} v_\sim$.
 The result follows by \refthm{thm:CKB}.
\end{proof}

\subsection{A Knuth-Bendix-like procedure on cyclic words}\label{sec:CS}

If a system $C(S)$, satisfying \ref{it:C1} and \ref{it:C2} above, 
 is large enough to ensure $u_\sim \CRAS{C(S)} v_\sim$ whenever 
$u_\sim$ and $v_\sim$ are conjugate in $\Gam^*/S$ with $u$ short, then 
we can apply \refthm{thm:CKB}; and we can use the system $C(S)$ for 
solving conjugacy in $\Gam^*/S$. In order to construct such a system 
we may use a form of Knuth-Bendix completion. This can be done in a very general way; which is fairly standard but technical, if we work out all details. 
Here 
we wish to restrict an analogue of  Knuth-Bendix completion to short words; for which we need some additional hypotheses.

We assume throughout this section that the alphabet $\Gam$ is well-ordered by $<$. 
We extend this well-order to the \emph{shortlex} order $<$ on $\Gam^*$ as usual: 
we write $u < v$ if either $\abs{u} < \abs v$ or $\abs{u} = \abs v$
and $u =p ax$, $u =p by$ with $a,b\in \Gam$ such that $a <b$. 
Moreover, we extend the well-order to cyclic words by representing 
a cyclic word $w_\sim$ by the minimal shortlex word in its class
$w_\sim = \set{uv}{vu = w}$. 
Hence, there is well-order on the set of cyclic words. 
For any relation $R\subseteq \Gam^*\times \Gam^*$ we define the \emph{descending part}
of $R$ to be 
\[\widetilde{R}=\{(l,r)\in R: l> r \textrm{ in the shortlex ordering}\}.\]

The new restriction we put on $S$ is that we assume that, for all $(\ell,r)\in S$, 
we have $\abs r 
\leq \abs \ell$. In particular, if $w$ is short and $w \CRAS S v$, then $v$ is short, too. 
Now let $C(S)$ satisfy \ref{it:C1} and \ref{it:C2} above. 
 We say that $(u_\sim, v_\sim)\in C(S)$ is a {\em short critical pair}, if $u_\sim > v_\sim$ (in the shortlex ordering) and for some $S$-short word $w$ we have:  
\begin{equation}\label{eq:criticalpair}
u_\sim \LAS{\circ}{C(S)} w_\sim \RAS{\circ}{C(S)}v_\sim
\end{equation}
We say that the critical pair  $(u_\sim, v_\sim)$ in (\ref{eq:criticalpair}) is \emph{shortlex resolved},
if 
\[u_\sim \CRAS{\widetilde{C}(S)} t_\sim \CLAS{\widetilde{C}(S)}v_\sim,\]
for  some $t$ with $v \geq t$ (where $\widetilde{C}(S)$ is the descending part of $C(S)$).
By {\em resolving} $(u_\sim, v_\sim)$ we mean adding the rule  $(u_\sim,v_\sim) $ to $C(S)$. (Note that, by definition, $u_\sim > v_\sim$.) Hence by resolving we force $(u_\sim, v_\sim)$  to be resolved. 

If we begin by taking $C(S)$ equal to $\RAS{\circ}{S}$ then, by resolving short pairs, 
we may form new systems which still satisfy 
\ref{it:C1} and $\ref{it:C2}$. 
If the alphabet $\Gam$ is finite, then this procedure of adding more rules terminates 
because there are only finitely many short words. In general,  there exists a limit system $C^*(S)$, satisfying
\ref{it:C1} and $\ref{it:C2}$ and 
such that all 
short critical pairs are shortlex resolved, but if $\Gam$ is infinite then we may only have a 
 semi-procedure for its construction.  

\begin{theorem}\label{thm:CKB_2}
Let $S \sse \Gam^* \times \Gam^*$ be a standard strongly confluent semi-Thue system such that 
$\GS$ is a group and such that, for all $(\ell,r)\in S$, 
we have $\abs r 
\leq \abs \ell$. Let $C^*(S)$  be constructed as above by resolving short critical pairs. Then the following two assertions hold: 
\begin{enumerate}[1.]
\item  The system $C^\ast(S)$ is standard and  confluent.  
\item Two words $u$ and $v$ are conjugate in $\GS$ \IFF 
there exists a cyclic word $t_\sim$ such that 
 \[u_\sim \CRAS{C^\ast(S)} t_\sim \CLAS{C^\ast(S)}v_\sim.\]
\end{enumerate}
\end{theorem}

\begin{proof}
  By construction  $C^\ast(S)$ is standard. Having shortlex resolved all short critical pairs, 
  the descending part $\widetilde C = \widetilde{C^*}(S)$ of $C^*(S)$ %
  is terminating and 
  contains all new rules $(u_\sim,v_\sim)$ added to the system $\RAS{\circ}{S}$.
  Therefore $\widetilde C$ is  locally confluent on short words. 
  Moreover, if $u_\sim \RAS{\circ}{S} v_\sim$ then,
  since $<$ is a total order and $|l|\ge |r|$ for all $(l,r)\in S$, 
  either $(u_\sim,v_\sim)$ or $(v_\sim,u_\sim)$ belongs to $\widetilde C$. Hence   
  $u_\sim \CDAS{C^\ast(S)} v_\sim $ \IFF 
  $u_\sim \CDAS{\widetilde C} v_\sim $. 
  (Note that we don't claim that $\widetilde C$ satisfies C1 or C2. 
  We don't even have $S \sse \widetilde C$, in general.) 
  
  We are now ready to show that $C^\ast(S)$ is confluent on short words. 
  Consider the following  situation where $w$ is short: 
  $$u_\sim \CLAS{C^\ast(S)} w_\sim \CRAS{C^\ast(S)}v_\sim.$$
  
  Since $\abs r \leq \abs \ell$ for $(\ell,r)\in S$
  (and hence for all $(\ell,r)\in C^\ast(S)$) 
  we see that 
  $u$ and $v$ are short, and moreover
  $u_\sim \CDAS{\widetilde C} v_\sim $. Note that the path
  $u_\sim \DAS{}{\widetilde C} u_\sim'  \DAS{}{\widetilde C} \cdots \DAS{}{\widetilde C}v_\sim $ 
  ( via $w_\sim$) never leaves the set of short words. 
  Being terminating and locally confluent, the system 
  $\widetilde C$ is confluent on short words. Hence,  since $u_\sim \CDAS{\widetilde C} v_\sim $, 
  there exists a cyclic word $t_\sim$ such that 
  \[u_\sim \CRAS{\widetilde C} t_\sim \CLAS{\widetilde C}v_\sim.\]
  As 
  $\widetilde C \sse C^\ast(S)$, we see that $C^\ast(S)$ is confluent on short words, as claimed. 
  
  Finally, $C^\ast(S)$ satisfies conditions C1 and C2 above; 
  and so \refcor{cor:csconj} applies, to give the result. 
\end{proof}
\subsection{Strongly confluent Thue systems}
\label{sec:sct}

If the system $S$ is Thue (c.f. Section~\ref{sec:ts}) then we may construct $C^*(S)$ in finitely many
steps as follows. We start with  $C_0 = C_0(S)\,= \;  \RAS{\circ}{S}$.
This is a relation defined on the set of  cyclic words where all rules are either  length decreasing   
or length preserving and then symmetric. We  call any such relation on cyclic words  \emph{Thue}.

At each step let us define a Thue relation $C_i$ 
satisfying conditions C1 and C2 above. 
We let $U_i$ be the set of ``unresolved
short critical pairs'' $(u_\sim,v_\sim)$, which are defined in the Thue case as follows: 
$$ u_\sim \LAS{\circ}{C_i} w_\sim \CRAS{C_i}w'_\sim \RAS{\circ}{C_i} v_\sim$$
where $w$ is $S$-short, $\abs w = \abs w'\geq \abs u \geq \abs v \geq 1$,  
and neither $ u_\sim \CRAS{C_i} v_\sim$ nor $ u_\sim \CLAS{C_i} v_\sim$.

Note that, since $\abs w = \abs w'$ we have $ u_\sim \LAS{\circ}{C_i} w_\sim \CLAS{C_i}w'_\sim \RAS{\circ}{C_i} v_\sim$, too. Thus, for unresolved pairs 
we must have $\abs w >  \abs u \geq \abs v \geq 1$. (Because if, say $\abs w = \abs v$, then $ u_\sim \CLAS{C_i}w'_\sim \LAS{\circ}{C_i} v_\sim$.)

At  the next step we let $C_{i+1}$ be the relation obtained from $C_i$ by adding a pair
$(u_\sim, v_\sim)$ to $C_i$, for all $(u_\sim,v_\sim)\in U_i$, and, in
addition, by adding $(v_\sim, u_\sim)$ whenever  
$|u_\sim|= |v_\sim|$. This keeps $C_{i+1}$ Thue. Finally, 
 we let
\begin{equation}\label{eq:KB2} 
C^*(S)=\bigcup \set{C_{i}(S)}{i \in \NN}.
\end{equation}

\begin{theorem}
Let $S$ be a standard, strongly confluent, Thue system, let $m=m(S)$ and let $C^*(S)$ be the
system defined in \eqref{eq:KB2} above. Then 
$C^*(S)=C_{2m-2}$, and $C^*(S)$  is a confluent, Thue system, satisfying conditions \ref{it:C1} and \ref{it:C2}.  
\end{theorem}
\begin{proof}
By definition
$C_i$ are Thue for all $i\ge 0$; and short words have length at most  $2m-2$.
When considering $ u_\sim \LAS{\circ}{C_i} w_\sim \CRAS{C_i}w'_\sim \RAS{\circ}{C_i} v_\sim$
we may assume that $\abs u < \abs w$ (see above) and that $u_\sim \LAS{\circ}{C_i\sm C_{i-1}} w_\sim $ (or $w'_\sim \RAS{\circ}{C_i\sm C_{i-1}} v_\sim$). Thus, at every step the words $w$ under consideration get shorter. We 
conclude $C^*(S)=C_{2m-2}(S)$, as claimed.  

Next, we show that $C^*(S)$ is confluent on short cyclic words.
To this end we define an equivalence relation $\equiv$ on cyclic words by
 $ u \equiv  v$ if 
$ u_\sim \CRAS{C^*(S)} v_\sim$ and $ u_\sim \CLAS{C^*(S)} v_\sim$. 
Thus, if $ u \equiv  v$ then  $ u_\sim \CDAS{C^*(S)} v_\sim$ and $\abs u = \abs v$.
We can view  $C^*(S)$ as a terminating rewriting system on equivalence classes $[u] =  \set{v}{v \equiv  u}$. 
By construction, $C^*(S)$ is locally confluent on classes $[w]$, where $w$ is short. But together with termination, we see that $C^*(S)$ is actually confluent on these classes $[w]$. But this implies that $C^*(S)$ is confluent 
on short cyclic words, because it is Thue. 
Finally, $C^*(S)$ satisfies the two conditions C1 and C2 above. Since 
$S$ is also  a standard, strongly confluent semi-Thue-system, we may apply 
\refthm{thm:CKB}.
\end{proof}

\subsection{Cyclic geodesically perfect systems}\label{sec:cgp}
In this section we consider  an analogue for cyclic rewriting systems
of geodesically perfect string rewriting systems; and 
adapt our Knuth-Bendix completion process to these systems.      
Let $S \sse \Gam^* \times \Gam^*$ be a standard semi-Thue system such that 
$\GS$ is a group. 
A cyclic word $w_\sim$ is called \emph{geodesic} (w.r.t.{} $S$), if 
 $w$  is a  shortest word in its conjugacy class. That is  
 \[\abs w = \min\set{\abs u}{u \in \Gam^*\textrm{ and }\exists x: xu \oi x = w\in \GS}.\]
 A cyclic word $w_\sim$ is called \emph{quasi-geodesic} (w.r.t.{} $S$), if 
 it is either geodesic or it is strictly $S$-short, but it is not equal to the neutral element in 
 $\GS$. Note that all non-trivial geodesic cyclic words are  quasi-geodesic 
  and more importantly 
 in 2-monadic systems every quasi-geodesic cyclic word is actually geodesic. 

Now,  a  Thue relation $C(S)$ on cyclic words, 
satisfying \ref{it:C1} and \ref{it:C2} above, is called \emph{quasi-geodesic}, 
if by applying a sequence of  
length reducing rules from $C(S)$  to a cyclic word $w_\sim$ 
we eventually derive a  quasi-geodesic  
cyclic word $u_\sim$.   
In order to  be geodesically perfect $C(S)$ must satisfy  stronger conditions:
 $C(S)$ is called   
 \emph{geodesically perfect}  if, by applying  a sequence of 
 length reducing rules from $C(S)$ to a cyclic word $w_\sim$,  
 we eventually derive a geodesic  
 cyclic word $u_\sim$.  Moreover, if  two geodesics $u_\sim$ 
and $v_\sim$ can both  be derived  from $w_\sim$, then it must be possible to rewrite $u_\sim$ into 
$v_\sim$ using only  length  preserving rules from $C(S)$. 
Note that every geodesically  perfect Thue system on cyclic words is confluent.

 Now, if $S\subseteq \Gam^*\times \Gam^*$ is a Thue system then we 
 say that $S$ is \emph{C-quasi-geodesic}  
 if the system $\RAS{\circ}{S}$, on cyclic words,  
 is quasi-geodesic. 
The following result shows that a geodesic Thue system  is innately 
 C-quasi-geodesic. 
 \begin{theorem}\label{thm:qcg}
 Let $S \sse \Gam^* \times \Gam^*$ be a standard,   
 geodesic, Thue system. Then $S$ is C-quasi-geodesic. 
\end{theorem}  
\begin{proof}  
We have to show  the following: 
if $u_\sim \CDAS{S} w_\sim$ and $\abs u < \abs w$, then either a 
length reducing rule applies to the cyclic word $w_\sim$ or 
$w_\sim$ is strictly $S$-short. To begin with  
let $u_\sim \CDAS{S} w_\sim$. 
Then there is a sequence  $u = w_0, \ldots, w_\ell=w$ 
such that $w_{i-1}$ and  $w_{i}$ 
are related in one of the following three ways: 
 \[w_{i-1}\RAS{}S w_{i} \quad \mbox{or} \quad w_{i-1}\LAS{}S w_{i}
  \quad \mbox{or} \quad w_{i-1}\sim w_{i}.\]
 First, we claim that there exist  $m\in \NN$ and $u_1,u_2 \in \Gam^*$ such that $u_1 u^{k-m}u_2 \DAS{*}S w^k$, for all $k> m$.

 This is true for $\ell=0$ with $m=0$. For $\ell \geq 1$ the result holds
 by induction for 
 $v= w_1, \ldots, w_\ell=w$ with some $m\in \NN$ and $v_1,v_2 \in \Gam^*$. 
 Now, if $u\RAS{}S v$, then we have 
$v_1 u^{k-m}v_2\RAS{*}S v_1 {v}^{k-m}v_2 \DAS{*}S w^k$, for all $k> m$. Similarly, if $u\LAS{}S v$, then we have $v_1 u^{k-m}v_2\LAS{*}S v_1 {v}^{k-m}v_2 \DAS{*}S w^k$, for all $k> m$.
 Now, let $u= u_2u_1$ and $v = u_1u_2$. Define $m' = m+1$. 
 We have:   
 \[v_1u_1 {u}^{k-m-1}u_2v_2  \DAS{*}S v_1 {v}^{k-m}v_2 \DAS{*}S w^k, \text{ for all } k> m.\]
 Replacing $m$, $u_1$ and $u_2$ with $m'$,  $v_1u_1$, and $u_2v_2$,
respectively,  we see that the claim holds. 

  Next, assume that we have $\abs u < \abs w$ and choose $m$, 
$u_1$ and $u_2$ as above. Take $k$ large enough  to make 
  $|w^k|>|u_1 u^{k-m}u_2|$.  Since $u_1 u^{k-m}u_2 \DAS{*}S w^k$ and $S$ is geodesic, a length reducing rule $(\ell, r)\in  S$ applies to  $w^k$. 
  If $\abs \ell \leq \abs w$, then the same rule applies to the cyclic word $w_\sim$, and we are done. In the other case, $w$ is strictly $S$-short, and we are done, too.  
  \end{proof}
 
In the next section of the paper we shall be concerned with standard,  
geodesically perfect, Thue string rewriting systems $S$, which are $2$-monadic:
 that is $m(S)=2$. For the rewriting system $\RAS{\circ}{S}$ induced by
such $S$, there is a particularly simple form
of Knuth-Bendix completion. In this case we consider an short critical
pair $(u_\sim,v_\sim)$ to be ``unresolved'' if it arises from the 
situation 
\begin{equation}\label{eq:gpscp}
u_\sim \LAS{\circ}{S} w_\sim \RAS{\circ}{S}v_\sim,
\end{equation}
where $w$ is short and $\abs {w} > \abs {u} \geq \abs {v} \geq 1$.
We \emph{resolve} the short critical pair of \eqref{eq:gpscp} by
adding the rules  
$(u_\sim,v_\sim)$ and  $(v_\sim,u_\sim)$. 
Let  $C^\dag(S)$ be the system obtained from $\RAS{\circ}{S}$ by
resolving all short critical pairs of the form \eqref{eq:gpscp}.  
 Note that if $(u_\sim, v_\sim)$ is a short critical pair then
both $u$ and $v$ are strictly short and non-trivial so, $S$ being $2$-monadic,
we have
$|u|=|v|=1$. 
\begin{corollary}\label{cor:gpercyc}
Let $S \sse \Gam^* \times \Gam^*$ be a standard, $2$-monadic, 
 geodesically perfect, Thue system, such that $\GS$ is a group, and 
$C^\dag(S)$ is confluent. 
Then $C^\dag(S)$ satisfies 
\ref{it:C1} and \ref{it:C2} and  is geodesically perfect. 
Moreover two cyclic words $u_\sim$ and $v_\sim$ are conjugate in $\GS$
if and only if there exists a cyclic word $t_\sim$ such that 
\[u_\sim\CRAS{C^\dag(S)}t_\sim\CLAS{C^\dag(S)}.\] 
\end{corollary}
\begin{proof}
By construction $C^\dag=C^\dag(S)$ satisfies \ref{it:C1} and \ref{it:C2}. Two elements
$u,v\in \Gam^*$ are conjugate if and only if $u_\sim \CDAS{C^\dag} v_\sim$; 
so the final statement holds if $C^\dag$ is confluent. Therefore
it is sufficient to prove that $C^\dag$ is geodesically perfect. 
 
 Consider $w\CRAS{C^\dag}v$ such that $v$ 
has minimal length with this property (so is geodesic)
 and  let $w\CRAS{C^\dag} u $ be some maximal derivation using only length reducing rules from the cyclic rewriting system $C^\dag$. Clearly, 
 $\abs u \geq \abs{v}$; and Theorem \ref{thm:qcg} implies that 
$S$ is C-quasi-geodesic so either $|u|=|v|$ 
or $u$ is strictly $S$-short.  We have to show that we can transform 
 $u$ into $v$ by length preserving rules from $C^\dag$. This is clear, if
 $v$ is not strictly $S$-short, because then $|u|=|v|$, and $C^\dag$ is 
 confluent and Thue.  For $m(S) = 2$, a strictly $S$-short 
 word $v$ is either a letter or the empty word $1$. But if $v=1$ we have $w\RAS*S v$
 because $S$ is a confluent semi-Thue system and $1$ is irreducible
; and it follows from the definitions of $\CRAS{C^\dag}$ and $u$ that $u=1$ as well.  There remains the case 
 $v\in \Gam$. Since $S$ is C-quasi-geodesic we have $\abs u =1$, too. 
 As $C^\dag$ is confluent and Thue we can transform the letter 
$u$ into  $v$, by
applying length preserving rules of $C^\dag$. 
\end{proof} 

\section{Stallings' pregroups and their universal groups}
\label{sec:pregroup}

We now turn to the notion of pregroup in the sense of Stallings,
\cite{Stallings71}, \cite{Stallings87}.
A \emph{pregroup} is a set $P$ with a distinguished element
$\eps$,
equipped with  a partial multiplication 
 $(a,b) \mapsto ab$ which is defined for $(a,b) \in D$, where $D \subseteq P \times
P$, and an involution 
$a \mapsto \ov{a}$, satisfying the following axioms,
for all $a,b,c,d \in P$. (By ``$ab$ is defined'' we mean  that
$(a,b)\in D$.)
\begin{enumerate}[(P1)]
\item  
$a\eps$ and $\eps a$ are defined and
$a \eps = \eps a = a;$
\item  
$\ov{a} a$ and $a \ov{a}$ are defined and $ \ \
\ov{a} a = a \ov{a} = \eps;$
\item  
if $ a b$ is defined, then so is
$\ov{b} \ov{a},$ and
$\ov{a  b} = \ov{b}\,  \ov{a};$
\item  \label{it:P4}
if $a  b$ and $b
c$ are defined, then $(a b) c$ is defined if and only if $a (b
c)$ is defined, in which case
\[ (a b) c =  a (b c);\]
\item   \label{it:P5} 
if $a  b, b  c,$ and $c  d$
are all defined then either $a  b  c$ or $b
c  d$ is defined.
\end{enumerate}
It is shown in \cite{Hoare88}
that (P3) follows from (P1), (P2), and (P4), hence can be omitted.

For $a,b\in P$ we write $ab\in P$, to mean that $ab$ is defined. Also we
use the notation $[ab]$ 
 to indicate that $ab\in P$ and, under the partial multiplication, 
$(a,b) \mapsto [ab]$. This notation is extended recursively to 
products of more than
two elements of $P$: if $w\in P^*$, where the notation has been established
for words shorter than $w$,  and 
$w$ has a factorisation $w=uv$, such that 
$u, v\in P$  and $[u][v]$ is defined, we write $w\in P$ and 
use $[w]$ to denote  the product $[u][v]\in P$. 
(Note though that, for example, $[abc]$ means only that one of  $[ab]c$ 
or $a[bc]$ belongs to $P$. (\emph{cf.} \reflem{lem:key}.))

The set $P$ can be considered as a possibly infinite alphabet. 
The axioms above lead to the following definitions of Thue systems
 $S_\eps$, $S(P)$ and the universal group $U(P)$.

\begin{definition}\label{def:ug}
\begin{enumerate}
\item The system $S_\eps\subseteq P^* \times P^*$ is defined by the following rules: 
\[
\begin{array}{rcll}
\eps & \longrightarrow  & 1 &\mbox{(= the empty word)}\\
ab & \longrightarrow  &[ab] & \mbox{if } \; (a,b) \in D  \\
ab & \longleftrightarrow  &[ac][\ov c b]& \mbox{if }
  \;(a,c), \, (\ov c,b) \in D
\end{array}
\]
\item Let $\Gam = P \sm \oneset{\eps}$. 
The system  $S(P) \subseteq 
\Gamma^* \times \Gamma^*$ is  defined as follows: 
\[
\begin{array}{rcll}
ab & \longrightarrow & 1  & \mbox{if } \; (a,b) \in D \mbox{ and } [ab]=\eps.\\
ab & \longrightarrow  &[ab] & \mbox{if } \; (a,b) \in D  \mbox{ and } [ab] \neq \eps.  \\
ab & \longleftrightarrow  &[ac][\ov b]& \mbox{if }
  \;(a,c), \, (\ov c,b) \in D, \mbox{ and } (a,b) \notin D. 
\end{array}
\]
We say that $S(P)$ is the Thue system \emph{associated with $P$}. 
\item The \emph{universal group} $U(P)$ of a pregroup $P$ is the group
\[U(P)=\Gam^*/\set{\ell = r}{(\ell,r) \in S(P)}. \]
\end{enumerate}
\end{definition}
Tietze transformations may be applied to the presentation $P^*/S_\eps$ to
give the presentation $\Gam^*/S(P)$; so $U(P)\cong P^*/S_\eps$. 

A \emph{reduced word} is an 
element $p_1\cdots p_n$ of $P^*$ 
 such that all $p_i\in \Gam$ and  $[p_ip_{i+1}]
\notin P$, for $i$ from $1$ to $n-1$. 

The relationships between a pregroup, these rewriting systems
and the  universal group
rest on several key lemmas, the most important of which we 
restate here for completeness. 

\begin{lemma}[{\cite{Stallings71}}]\label{lem:key}
Let $a,b,c,d, g,h \in P$. 
\be[1.)]
\item\label{it:keyi} If $ab\in P$ then $[ab]\ov b\in P$ and $[ab\ov b]=a$. 
\item\label{it:keyii} If $ab\notin P$ but $ac$ and $\ov c b\in P$ then 
$[ac][\ov c b]\notin P$.
\item\label{it:keyiii} If $abc$ is a reduced word and $a \ov d,db\in P$ then
$[a \ov d][d b]c$ is a reduced word. 
\item\label{it:keyiv} If $ab\notin P$ but $ac$, $\ov c b$, $bd\in P$ then
$\ov c bd\in P$. (That is $[\ov cb]d\in P$ from which it 
follows that $[\ov c[bd]]=[[\ov cb]d]$.) 
\item\label{it:keyv}
If 
$gb,\ov b h, gbc,\ov c \ov b h \in P$, but $gh \not \in P$ 
then $bc \in P$. 
\ee
\end{lemma}

\begin{proof}
\be[1.)]
\item Apply (P4) to the triple $a, b, \ov b$.
\item Use \ref{it:keyi} and apply (P4) to the 
triple $[ac], \ov c$ and $b$.
\item From the above $[a \ov d][db]$ is reduced and $\ov d db\in P$. 
If $dbc\in P$ 
then consider the four element product $a \ov d [d b] c$. From 
 (P5), either $ab\in P$ or $b c\in P$, a contradiction. 
\item Consider the four elements $[ac]$, $\ov c$, $b$ and $d$, of $P$. 
The product of each adjacent pair is defined, so (P5) implies
either $ab=[ac][\ov c b]\in P$, or $\ov c b d\in P$. 
\item
 Consider the product of four elements $\ov g [gb] c [\ov c \ov b h]$. 
 By hypothesis we have $ gbc, \ov b h \in P$. Moreover, 
  $[gb] c [\ov c \ov b h]= gh \notin P$. Hence, by (P5) we conclude 
  $\ov g [gb] c = [bc] \in P.$
\ee 
\end{proof}
As a consequence of Lemma \ref{lem:key}.\ref{it:keyiii} and \ref{it:keyiv} 
the set of reduced words coincides with the set of $S(P)$-geodesic and
 the set of $S_\eps$-geodesic words.  

The length preserving rule
$\longleftrightarrow$ of $S_\eps$ is the length $2$ case of Stallings'
\emph{interleaving}  relation $\approx$
defined on     
words in $P^*$ as follows.  
If $a_i, c_i\in P$, for $i=1, \ldots n$, and $\ov c_{i-1} a_i$, $a_ic_i$ and $\ov c_{i-1} a_i
c_i\in P$ with $c_0=c_n=\eps$, then 
\[a_1  \cdots a_n \approx b_1 \cdots b_n,\]
where $b_i=[\ov c_{i-1}a_ic_i]$.
Stallings used \reflem{lem:key} to show that interleaving is an
equivalence relation on reduced words
 and 
this equivalence relation is central to the proof of 
Theorem \ref{thm:stall}.\ref{stallii} in \cite{Stallings71}.  
Another approach is taken in \cite{ddm10}, based on  the following lemma,
which is again proved using \reflem{lem:key}.  

\begin{lemma}\label{lem:sconfl}
The Thue system $S_\eps$ is strongly confluent.
\end{lemma}

Parts 
\ref{stallii} and \ref{stalliii} of the following theorem are from  Stallings \cite{Stallings71}. Part \ref{stalliv}  is from  
\cite{ddm10}. 

\begin{theorem}[\cite{Stallings71},\cite{ddm10}]\label{thm:stall}
Let $P$ be a pregroup. Then the following hold.
  \begin{enumerate}[1)]
 \item\label{stallii} $P$ embeds into $U(P)$.
 \item\label{stalliii} If $g$ and $h$ are reduced words  $P^*$ then
$g=_{U(P)} h$ \IFF $h$ is an interleaving of $g$. 
  \item\label{stalliv} $S(P)$ is a geodesically perfect Thue system.
\end{enumerate}
\end{theorem}

\begin{proof}
\ref{stallii} is a  direct consequence of \reflem{lem:sconfl} 
and the remark following the proof of Proposition \ref{prop:grim}. \ref{stalliii} follows from  \ref{stalliv}
and \reflem{lem:key}. 
 The proof of \ref{stalliv} is given in  \cite{ddm10}: 
 however, 
 for completeness we give a proof.
  Consider 
  a word $u = a_1 \cdots a_n$ with $a_i \in \Gam$ such that $a_ia_{i+1}$ is not defined in $P$ for $1\leq
i<n$. Assume that after  a sequence of applications of symmetric rules, we can apply a 
length reducing one.  We have to show that some length reducing rule applies 
to $u$. We may assume that the sequence of applications of symmetric rules
is not empty, but as short as possible.  
 The corresponding word contains
a factor $abcd$ with $a,b,c,d \in \Gam$ and neither $ab$, $bc$  nor $cd$ 
 defined in $P$. Applying the last symmetric rule yields  
$a[b \ov x] [x c]d$. The length reducing rule 
cannot then apply to $[b \ov x] [x c]$, since this is not defined, by
\reflem{lem:key}.\ref{it:keyii},  and so must apply to $a[b \ov x]$ or 
$[x c]d$. In both cases we have a contradiction to 
\reflem{lem:key}.\ref{it:keyiii}. 
\end{proof}

\begin{remark}\label{rem:frida}
Every group $G$ is the universal pregroup of some 
pregroup $P$. Indeed, $G = U(G)$. 
Moreover, \refthm{thm:stall} tells us that  every pregroup $P$ can be defined as a subset $P \sse G$ inside a group 
$G$ such that $1 \in P$, $a\in P$ implies $\oi a \in P$, and $P$ satisfies the
axiom (P5). Having such a subset the domain $D$ becomes $D = \set{(a,b) \in P\times P}{ab \in P}$. 
\end{remark}

\subsection{Amalgamated products and HNN-extensions} \label{hnnag}
The guiding example of an universal group in the sense of Stallings is the amalgamated product
$G= A\ast_H B$ of two groups over a common subgroup $H = A\cap B$. 
In this case $P= A\cup B$ forms a pregroup with 
$U(P) = G$. In this case, for  $a,b \in P$, the product $ab$ is defined in $P$ 
if and only if 
$a, b \in A$ or $a, b \in B$. The verification of (P5) is straightforward. 

The other obvious 
example of an universal group  is the case where $G = \HNN(H, t;\;  t^{-1}At =B)$ is an HNN-extension over 
two isomorphic subgroups $A,B$ in some base group $H$. (That is
there is an isomorphism $\phi:A\longrightarrow B$ and ``$t^{-1}At =B$'' denotes
the set of relations of the form $t^{-1}at=a\phi$, for all $a\in A$.) In this case we can choose $P= H \cup Ht^{-1}H \cup HtH$. 
Again, the verification of (P5) is straightforward. 

\subsection{Fundamental groups of graph of groups} \label{gog}
The notion of the  fundamental group of a graph of groups generalises 
amalgamated product and
HNN-extension to a much broader class. The concept of a \emph{graph of groups}
is due to Serre and the development of Bass-Serre theory has been a major achievement in modern group theory. We refer to the books  \cite{serre80},
\cite{bogopol08}, and to 
\cite{Rimlinger87a} for the background.

A \emph{virtually free group} is a group $G$ having a free subgroup of finite index. They are related to graphs of groups as follows. 
\begin{proposition}
\label{prop:grim}
Let $G$ be a finitely generated group. The following conditions are
equivalent.
\begin{enumerate}
\item $G$ is the fundamental group of a finite connected graph of groups where 
all vertex groups are finite. 
\item $G$ is the universal group of some finite pregroup.
\item $G$ can be presented by some finite geodesic system.
\item $G$ is virtually free.
 \end{enumerate}
\end{proposition}

\refprop{prop:grim} is taken {}from  \cite[Cor. 8.7]{ddm10} and combines several results from the literature.  It follows from \cite{Rimlinger87a}, \cite{ddm10}, \cite{Karrass}, and \cite{ms83}.

\section{Conjugacy in  universal groups}\label{sec:cpug}
We shall apply \refthm{thm:CKB} and Corollary \ref{cor:gpercyc}  to the universal group of a pregroup and in 
particular to the conjugacy problem. 
For this we fix a pregroup $P$, we let $U(P)$ be its universal group; and 
denote by  $S_\eps$ and $S=S(P)$ the Thue systems of  \refdef{def:ug}.
Let $C^\dag(S_\eps)$ and $C^\dag(S)$ be the cyclic rewriting systems
defined by resolving short critical pairs in the sense of Section \ref{sec:cgp}.

A \emph{cyclically reduced word} is a cyclic word 
over $\Gam^*$
which 
 is geodesic with respect to the rewriting system $C^\dag(S)$. 
We also refer to words $w\in w_\sim$ as cyclically reduced
if $w_\sim$ is cyclically reduced. In particular all elements of $\Gam$
are cyclically reduced. 

\begin{lemma}\label{lem:cyclth}
Let $g\in P^*$ be a cyclically reduced word and let $h\in P^*$ be 
a word such that $h_\sim$ is obtained from $g_\sim$ by applying a sequence
of length preserving rules of $C^\dag(S_\eps)$. Then $h_\sim$ is 
cyclically reduced and $|h|=|g|$. 
\end{lemma}
\begin{proof}
By induction it is enough to prove the case where $h_\sim$ is obtained
from $g_\sim$ by applying a single rule. If $g\in P$ then 
$g\neq \eps$, as $\eps$ is  not cyclically reduced, so $g\in \Gam$,  
 and 
the result follows.

If $|g|=n\ge 2$ then there exists a word $g_1\cdots g_n\in g_\sim$ and an
element $c\in P$  
such that $g_ig_{i+1}\notin P$ for all $i$ (subscripts modulo $n$), 
 $g_1c\in P$, $\ov c g_2
\in P$ 
 and 
$f =[g_1c] [\ov c g_2]\cdots g_n\in h_\sim$. 
As $g_1\cdots g_ng_1\cdots g_n$ is reduced, it follows from
 \reflem{lem:key}, \ref{it:keyii} \& \ref{it:keyiii} that $f^2$ is reduced. Therefore
$f$, and so also $h$, is cyclically reduced, as required.
\end{proof}

This lemma suggests that cyclically reduced cyclic words under 
cyclic rewriting should play the
role of reduced words under  standard
rewriting. This works as expected,  with the exception of the behaviour
of words of length $1$.  From 
\refthm{thm:stall}, two elements of $\Gam$ are equivalent 
under $S$ only if they are equal in $\Gam$.
However this is not true of cyclic words of length $1$ and the system
$C^\dag(S)$, and we often have to treat words of length one 
separately in what follows. 

Let $u = a_1 \cdots a_n \in \Gam^*$ with $a_i \in \Gam$.
A \emph{cyclic permutation} of $u$ is any element of $u_\sim$. 
Thus, a cyclic permutation
is the same as a transposition in $\Gam^*$. 
Let $n \ge 2$. 
If for $i=1,\ldots n$, there are elements $b_i, c_i\in P$ such that 
$\ov c_{i-1} a_i$, $a_i c_i$ are in $P$,  and $b_i=[\ov c_{i-1} a_i c_i]$
(subscripts modulo $n$), then any
element of $v_\sim$, where 
$v=b_1\cdots b_n$ is called
a \emph{cyclic interleaving} of $u$; and  $u_\sim$ is also called a 
cyclic interleaving  
of $v_\sim$. 
A \emph{preconjugation}  
of $u$ by $c \in P$  (when $n\ge 2$) is the cyclic interleaving  
$v= [\ov c a_1] a_2 \cdots a_{n-1} [a_n c]$.

For $u \in \Gam$ (i.e., $n = 1$) a \emph{cyclic interleaving} 
 of $u$ by $c \in P$ is defined as $v= [cu \ov c]$ in case 
that  $cu \ov c \in P$ is defined. A \emph{preconjugation} is 
defined to be a cyclic interleaving in this case.

In all cases every cyclic interleaving of $u$ may be  obtained by 
a cyclic permutation, followed by an interleaving, followed by
a preconjugation. Moreover every cyclic interleaving
of $u$  is conjugate  to $u$ in $U(P)$. The following
lemma describes more precisely how these definitions are related.

\begin{lemma}\label{lem:cleave}
Let $g$ and $h$ be cyclically reduced words over $\Gam^*$. 
If $|g|=1$ then $h$ is a cyclic interleaving of  $g$ if and only if 
$h_\sim$ is obtained from $g_\sim$ by applying a 
length preserving rule from $C^\dag(S)$.
If $|g|\ge 2$, then
the following are equivalent. 
\be[1.)]
\item\label{it:cleavei} $h_\sim$ is obtained from $g_\sim$ by the application of a 
finite sequence of length preserving rules from $C^\dag(S)$.
\item\label{it:cleaveii} There exists a word $f$, obtained from $g$ by a cyclic permutation
 followed by a single preconjugation, such that $h=_{U(P)} f$. 
\item\label{it:cleaveiii} $h$ is a 
cyclic interleaving of $g$. 
\ee
\end{lemma}

\begin{proof}
First consider the case $n=1$. 
Then $h$ is a cyclic interleaving  of $g$ if only if there
exists  $b\in P$
such that either  $\ov bg$ or $g b \in P$ and  $[ \ov bg b]=h \in P$.
On the other hand, there is  a symmetric rule in $C^\dag(S)$ 
transforming  $g_\sim$  to $h_\sim$ if and only if there exists
$b\in P$ such that either  $ \ov bg\in P$ and  $[\ov bg] b \RAS{}S h$
(in which case $b[\ov bg]\RAS{}S g$); or 
  $g b \in P$
and  $\ov b [gb] \RAS{}S h$. 

Now suppose $n\ge 2$. We show first that \ref{it:cleaveiii} implies
\ref{it:cleavei}. If \ref{it:cleaveiii} holds then there exist $g_i,a_i\in P$ such
that $\ov a_{i-1} g_i$, $g_i a_i$ and  
$\ov a_{i-1} g_i a_i\in P$   and 
$h$ is a cyclic permutation of $h_1\cdots h_n$, where $h_i=[\ov a_{i-1} g_i a_i]$. Therefore, we may successively
apply symmetric rules of $S$ to $g_\sim$ to obtain $(h_1\cdots h_n)_\sim=h_\sim$
as required. 

Next we show that \ref{it:cleavei} implies
\ref{it:cleaveii}. If   \ref{it:cleavei} holds then there exist
words $g_0,\ldots, g_n$ in $\Gam^*$ such that $g_0=g$, $h\in g_{n\,\sim}$
and $g_{i+1\,\sim}$ is obtained by applying a symmetric rule of $C^\dag(S)$ to 
$g_{i\,\sim}$. 
If $n=0$ then $h$ is a cyclic permutation of $g$ and there is nothing
further to do. Assume then that $n>0$. 
From \reflem{lem:cyclth} $g_i$ is cyclically reduced for all $i$. 
By definition there exists  a word $g_0=a_1\ldots a_n\in g_\sim$ and 
an element $c\in P$ such that $a_1 c\in P$, $\ov c a_2\in P$ and
$g_1=b_1\cdots b_n$, where $b_1=[a_1 c]$, $b_2=[\ov c a_2]$ and 
$b_i=a_i$, for $i\ge 2$. 
By induction, there exists a word $f_1$, obtained from $g_1$ 
 by a cyclic permutation
 followed by a single preconjugation, such that $h=_{U(P)} f_1$. 
There are several cases to consider, depending on which cyclic permutation
of $g_1$ is taken. Assume $f_1$ is a preconjugation of a cyclic permutation 
$b_{i+1} \cdots b_i$ of $g_1$, where $0\le i\le n-1$. That is, 
there exists $d\in P$ 
such that  $\ov d b_{i+1}\in P$, $b_i d\in P$ and 
$f_1=c_{i+1}\cdots c_i$, where $c_{i+1}=[\ov d b_{i+1}]$, $c_i=[b_i d]$
and $c_j=b_j$, if $j\neq i,i+1$. Thus
\[
f_1=
\begin{cases}
c_1c_2=[\ov d[a_1c]][[\ov ca_2 ] d] & \textrm{ if } i=0 \textrm{ and } n=2\\
c_1c_2\cdots c_n= [\ov d[a_1 c]][\ov c a_2] \cdots [a_n d]&  \textrm{ if } i=0 \textrm{ and } n\ge 3\\
c_2c_3 \cdots c_nc_1=[\ov d [\ov c a_2]] a_3 \cdots a_n [[a_1 c]d]&  \textrm{ if } i=1 \\
c_3\cdots c_nc_1c_2=[\ov d a_3]\cdots  a_n[a_1c][[\ov c a_2]d] &  \textrm{ if } i=2 \\
c_{i+1}\cdots c_nc_1c_2 \cdots c_i=[\ov d a_{i+1}] \cdots a_n
[a_1 c][\ov c a_2] \cdots [a_i d]&  \textrm{ if } i\ge 3
\end{cases}
\]

If $i\ge 3$ then $n\ge 3$ and, 
as $i+1\le n$,  we have  $c_1c_2=_{U(P)}a_1a_2$ so 
\[h=_{U(P)}f_1 =_{U(P)} [d a_{i+1}]\cdots a_1a_2 \cdots [a_{i} d],\]
 a preconjugation of the cyclic permutation $a_{i+1}\cdots a_i$ of $g$.

If $i=2$ then  \reflem{lem:key}.\ref{it:keyiv} applied to
 the four elements $[a_1c]$, $[\ov c a_2]$, $\ov c$ and $d$, shows that 
$[c \ov c a_2 d]=[a_2 d]\in P$. Therefore $c_1 c_2=_{U(P)}[a_1 c][\ov c a_2 d]=_{U(P)}a_1[a_2 d]$ and 
\[f_1=_{U(P)}[\ov da_3]\cdots a_n a_1 [a_2 d],\] 
as 
required. 

If $i=1$ then from \reflem{lem:key}.\ref{it:keyv} it follows
that $cd\in P$ so 
\[f_1=_{U(P)}[\ov{(cd)} a_2]a_3 \cdots a_n[a_1(cd)],\] 
as 
required.

If $i=0$ and $n\ge 3$  
then the result follows, by symmetry, from the case $i=2$, 
leaving the case $i=0$ and $n=2$. 
As $g$ is cyclically reduced, 
$a_1a_2a_1a_2$ is a reduced word and therefore so is $[a_1c][\ov c a_2]
[a_1c][\ov c a_2]$. Hence $[a_1c][\ov c a_2]$ is cyclically reduced.
Applying \reflem{lem:key}.\ref{it:keyiv} to $[\ov c a_2]$, $d$,
$[a_1c]$ and $\ov c$, gives $\ov d a_1\in P$. Similarly $a_2 d\in P$ and
the result follows as before.

Finally, to show that \ref{it:cleaveii} implies \ref{it:cleaveiii}, suppose
that $h=_{U(P)}f$, where 
\[f=[\ov b g_1]g_2 \cdots g_{n-1} [g_n b],\] and  
\[g=g_{i+1}\cdots g_n g_1 \cdots g_i,\] for some $i$. Then, from
Lemma \ref{lem:cyclth}, $f$ is cyclically reduced and hence,
from \refthm{thm:stall}, $h$ is an interleaving of 
$f$. From Lemma \ref{lem:key}, it  follows that $h$ is 
a  cyclic interleaving of $g$.  
 \end{proof}

\begin{lemma}\label{lem:epscconfl}
The system $C^\dag(S_\eps)$ is confluent. 
\end{lemma}

\begin{proof}
 The system $S_\eps$ is standard and it is 
 strongly confluent by \reflem{lem:sconfl}. Thus, by \refthm{thm:CKB}
 it is enough to show that $C^\dag(S_\eps)$ is confluent on all short cyclic 
words. Thus 
we have to consider 
 the situation: 
 \begin{equation}\label{eq:scpe} 
d_\sim \CLAS{C^\dag(S_\eps)} w_\sim \CRAS{C^\dag(S_\eps)} e_\sim,
\end{equation}
where $w$ is short. 
 We must show 
that \[d_\sim\CRAS{C^\dag(S_\eps)} t_\sim \CLAS{C^\dag(S_\eps)} e_\sim,\] 
for some $t_\sim$. 
 As $w_\sim$ is a short cyclic word 
 we have $\abs {w_\sim} \leq 2$. 
 If  $w=1$ in $U(P)$, then $u \RAS*{S_\eps} 1$, for all
 $u \CDAS{C^\dag(S_\eps)} w$,  and 
 we may take $t_\sim=1$. 
Thus, we may assume $1 \leq \abs {w_\sim}$ and $w \neq 1 \in U(P)$. 
If $\abs {w_\sim}= 1$ then $w_\sim$ is cyclically reduced, 
since $w \neq \eps$. 
 Hence all rules involved in \eqref{eq:scpe} are symmetric and 
we may take $t_\sim=w_\sim$. Thus, from now on in the proof we may 
 assume $\abs {w_\sim}= 2$. 
 Since all length preserving rules in $\RAS{\circ}{C^\dag(S_\eps)}$ are symmetric, we 
 are done if $d\notin \Gam$ or $e\notin \Gam$. 
 Thus, as suggested by the notation we have $d,e\in \Gam$.
 Again, since length preserving rules are symmetric, we may assume that
 the situation is  
 \[d_\sim \LAS{\circ}{C^\dag(S_\eps)} w_{0\,\sim} \DAS{\circ}{C^\dag(S_\eps)} 
\cdots \DAS{\circ}{C^\dag(S_\eps)} w_{k\,\sim}\RAS{\circ}{C^\dag(S_\eps)} e_\sim,\]
where 
 all $w_i$ have length 2.
As $w_{0\,\sim}$ is not cyclically reduced, Lemma \ref{lem:cyclth} implies
that no  
$w_{i\,\sim}$ is cyclically reduced.
Hence, for all $i$ there exists $e_i\in \Gam$ such that 
 $w_{i\,\sim}\RAS{\circ}{C^\dag(S_\eps)} e_i\in \Gam$.  
It therefore suffices to show that if
\[ d_{\sim}\LAS{\circ}{C^\dag(S_\eps)} u_{\sim}\DAS{\circ}{C^\dag(S_\eps)} v_\sim\RAS{\circ}{C^\dag(S_\eps)} e_\sim,\]
where $\abs{u}=\abs{v}=2$ and $\abs{d}=\abs{e}=1$, then  
\[d_\sim\CDAS{C_1} e_\sim,\]
where $C_1$ is the length preserving part of $C^\dag(S_\eps)$. 
  We may assume that
  $u_\sim =(ab)_\sim$ with $a,b\in \Gam$, $[ab]=d\in \Gam$, 
 and that  there exists $c\in \Gam$ such that either 
 $v_\sim = ([ac][\ov cb])_\sim$ or $v = ([\ov c a] [b c])_\sim$.  If $v_\sim= ([ac][\ov cb])_\sim$ then 
\[ d_\sim\LAS{\circ}{C^\dag(S_\eps)}v_\sim \RAS{\circ}{C^\dag(S_\eps)} e_\sim,\]
so $d_\sim\DAS{\circ}{C_1} e_\sim$ 
and we are done. 

Assume then that  $v_\sim = ([\ov c a] [b c])_\sim$. If $ba\in P$ then
 \[ d_\sim\LAS{\circ}{C^\dag(S_\eps)}u_\sim \RAS{\circ}{C^\dag(S_\eps)} [ba]_\sim,\]
and so $d_\sim\DAS{\circ}{C_1} [ba]_\sim$. Also  
 \[ e_\sim\LAS{\circ}{C^\dag(S_\eps)}v_\sim \RAS{\circ}{C^\dag(S_\eps)} [ba]_\sim,\]
 so $e_\sim\DAS{\circ}{C_1} [ba]_\sim\DAS{\circ}{C_1} d_\sim$, as required.

Therefore we may assume that $ba\notin P$. Applying (P5) to
the elements $\ov c$, $a$, $b$ and $c$ we have $\ov c ab$ or $abc \in P$. 
Also, 
from Lemma \ref{lem:key}.\ref{it:keyiii}, 
$[bc][\ov ca]\notin P$. As $v_\sim$ is not cyclically reduced it follows that 
$[\ov c a][ bc]\in P$, so we have 
 \[ d_\sim\LAS{\circ}{C^\dag(S_\eps)}([\ov c ab] c)_\sim \RAS{\circ}{C^\dag(S_\eps)} [\ov c ab c]_\sim \quad\textrm{ or }\quad d_\sim\LAS{\circ}{C^\dag(S_\eps)}(\ov c [ab c])_\sim \RAS{\circ}{C^\dag(S_\eps)} [\ov c ab c]_\sim,\]
and in both cases 
\[d_\sim\DAS{\circ}{C_1} [\ov c ab c]_\sim.\]
Moreover, 
\[e_\sim\LAS{\circ}{C^\dag(S_\eps)} v_\sim \RAS{\circ}{C^\dag(S_\eps)} [\ov c ab c]_\sim,\]
so  $d_\sim\DAS{\circ}{C_1} [\ov c ab c]_\sim \DAS{\circ}{C_1} e,$ as required.
\end{proof}

Having established the confluence of $C^\dag(S_\eps)$ we may get rid of 
the letter $\eps$ and the rule $\eps\ra {} 1$. 
That is: we switch back to the system $S=S(P)$. 

\begin{theorem}\label{thm:mainpre}
Let $S \subseteq 
\Gamma^* \times \Gamma^*$ be the Thue system  
associated with $P$, c.f.{} \refdef{def:ug}. Then $C^\dag(S)$ is geodesically perfect.
\end{theorem}

\begin{proof}The system $S= S(P)$ is a standard 2-monadic Thue system. 
The confluence of $C^\dag(S)$  follows {}from \reflem{lem:epscconfl}. 
By \refthm{thm:stall} 
the semi-Thue system $S$ is geodesically perfect. The result 
follows  by \refcor{cor:gpercyc}.
\end{proof}

\begin{corollary}\label{cor:hugo} 
Cyclically reduced elements are minimal length representatives of their conjugacy class in $U(P)$. 
Let $g$ and $f$ be cyclically reduced elements of $\Gam^*$ such that 
$g$ is conjugate to $f$ in $U(P)$. Then the following hold. 
\begin{enumerate} 
\item\label{hugoi} $g$ and $f$ have the same length.
\item\label{hugoii} If $g\notin P$, i.e., $|g|\geq 2$, then we can transform the cyclic word $g_\sim$ 
into the cyclic word 
$f_\sim$ by a sequence of at most $\abs{g}$ length preserving rules from $C^\dag(S)$. 
\item\label{hugoiii}If $g\in P$, i.e., $|g| = 1$, then we can transform $g$ into
$f$ by a sequence of   preconjugations.
\end{enumerate}
\end{corollary}

\begin{proof}
 Immediate by the confluence of $C^\dag(S)$ and \reflem{lem:cleave}. 
\end{proof}

The following theorem is the main result in this section.  It  makes 
statement \ref{hugoii} of \refcor{cor:hugo}
much more precise. 

\begin{theorem}\label{thm:conjugation} 
Let $g$ and $f$ be a cyclically reduced elements of{} $\Gamma^*$ such that
$g$ is conjugate to $f$ in $U(P)$. Let $g = g_1 \cdots g_n$  with $g_i \in P$ and 
 $n=|g|\geq 2$. Then, we may obtain $f$, as an element in $U(P)$, by a single cyclic permutation followed by a  preconjugation. 
 More precisely, we have 
 \[f = [bg_{i}] \cdots g_n g_1\cdots [g_{i-1}\ov b] \in U(P),\]
  where $b\in P$ and $bg_{i}$, $g_{i-1}\ov b \in P$. 
 \end{theorem}
\begin{proof}
This follows directly from Corollary \ref{cor:hugo}.
\end{proof}

We may strengthen the statement of \refthm{thm:conjugation} for pregroups
which 
satisfy certain extra conditions. First, in any pregroup $P$ we can define a canonical subgroup
by
\[
G_P= \set{x\in P }{(x,y), (y,x) \in D,\;  \forall y \in P}.
\]

We say that $P$ satisfies the extra axiom (P6) if the following is true. 
\begin{equation}
(f,g) \notin D \wedge (f,\ov b ) \in D \wedge (b,g)\in D  \implies b \in G_P. 
\tag{P6}
\end{equation}

Axiom (P6) holds for the standard pregroups defining amalgamated products or HNN-extensions (as in Section \ref{hnnag}), but it does not hold in general for the pregroup defining the fundamental group of a graph of groups,  
as given in \cite{Rimlinger87a}.

\begin{remark}\label{rem:p6}
If $P$ satisfies the axiom (P6) then the element $b \in P$
in the statement of \refthm{thm:conjugation} is necessarily in the 
canonical subgroup $G_P$. 
\end{remark}

We say that $P$ satisfies the extra axiom (P7) if the following is true.

\begin{equation}
\begin{split}
(y,z)\in D  \wedge (x,[yz])\in D \wedge [yz]\notin G_P 
&\wedge s\in \{x, \ov x\} \wedge t\in \{ y, z \}\\
&\implies 
\{(s,t),(t,s)\}\subset D. \end{split}
\tag{P7}
\end{equation}
First note that (P7) implies axiom (P6). To see this, suppose that 
$b,f$ and $g$ satisfy the hypotheses of (P6). Let $z=\ov g$, $y=[bg]$
and $x=[f\ov b]$. Then $(y,z)\in D$, $[yz]=b$ and $(x,[yz])=(x,b)
=([f\ov b],b)\in D$. If $[yz]\notin G_P$ then (P7) gives $(x,y)\in D$
and this  implies that $(f\ov b,bg)\in D$, from which we infer
$(f,g)\in D$, a contradiction. Thus we must have $[yz]\in G_P$ and
(P6) holds.

Axiom (P7) holds for the standard pregroup defining an amalgamated product, but
not, in general,  
for the standard pregroup defining an HNN-extension (as in Section
\ref{hnnag}). 
In contrast the following axiom (P8) holds for  the standard pregroup 
of an HNN-extension, but
not for that of an amalgamated product.
\begin{equation}
(a,b)\in D  \wedge [ab]=c \implies a\in G_P \vee b\in G_P \vee
c\in G_P. \tag{P8}
\end{equation}
Again, axiom (P8) implies axiom (P6). Indeed, consider 
the condition $(b,g) \in D$. If we have $[bg]\in G_P$, then 
$(f,\ov b) \in D$ implies $(f,g) \in D$, contrary to the hypothesis 
of (P6). Given $(f,g)\notin D$, we can exclude $g \in G_P$. Thus, (P8) yields 
the implication of  (P6).

In pregroups
in which axiom (P7) holds, elements of $P$ behave well with
respect to preconjugation. 
\begin{lemma}\label{lem:p7}
Let $P$ be a pregroup satisfying axiom (P7), let $H= G_P$ be 
its canonical subgroup and let $a, b\in P$.
If $c\in P$ is a preconjugate of both $a$ and $b$ then 
either $c\in H$ or $b$ is a preconjugate of $a$.
\end{lemma}
\begin{proof}
Let $c=[\ov u a u]=[\ov v b v]$, for some $u,v\in P$. Then 
either $\ov u a\in P$ or $au\in P$.  Assume $\ov u a\in P$. 
We have $b=[v c\ov v]\in P$, so either $vc\in P$ or $c \ov v\in P$.
Assume $c=[[\ov u a] u]\notin H$. Then  $vc\in P$ together 
with (P7) implies $u\ov v\in P$. Similarly $c\ov v\in P$ implies
$u\ov v\in P$. By symmetry, if $au\in P$ and $c\notin H$ then
again $u\ov v\in P$. Therefore, either $c\in H$ or 
$b=[(v\ov u) a (u\ov v)]$, a preconjugate of $a$.  
\end{proof}

In pregroups
in which axiom (P8) holds, elements of  
$P\sm G_P$ behave well with
respect to preconjugation. 
\begin{lemma}\label{lem:p8}
Let $P$ be a pregroup satisfying axiom (P8) and $H= G_P$ its canonical subgroup. Let $a\in P\sm H$
and $b\in P$. 
\be[1.)]
\item If $b$ is a preconjugate of $a$ then $b=[\ov h a h]$, for some
element $h\in H$.
\item If $b$ is conjugate to $a$ then $b$ is a preconjugate of $a$. 
\ee
\end{lemma}
\begin{proof}
\be
\item If $b=[\ov c \ a c]$, where $c\in P$,  then either
$\ov c a\in P$ or $ac\in P$. By symmetry, assume $\ov c a\in P$. 
If  $c\notin H$ then (P8) implies $[\ov c a]=h\in H$, so $c=a\ov h$. 
Thus $b=[\ov c a c]=[h a\ov h]$, as required.
\item From \refcor{cor:hugo}.\ref{hugoiii} there exist a sequence of
elements $a=b_0,\ldots ,b_n=b$, of $P$, such that $b_{i+1}$ is a 
preconjugation of $b_i$, for all $i$. As each preconjugation is by
an element of $H$ it follows that $b$ is in fact a preconjugate of $a$. 
\ee
\end{proof}
\section{The conjugacy problem in amalgamated products and HNN-extensions}\label{sec:cpag}

\subsection{Conjugacy in amalgamated products}\label{cag}
As in Section \ref{hnnag}, the defining pregroup for the group $G = A*_H B$ can be chosen to be 
$P = A \cup B$; and  
the common subgroup $H$ is then equal to the canonical subgroup $G_P$. 
Therefore $P$ satisfies (P6). 
Conjugacy of elements of a free product with amalgamation is
described in 
\cite{mks66}, which now follows easily
from of \refcor{cor:hugo} and \refthm{thm:conjugation} as we show below.
First we state the theorem.
\begin{theorem}[\cite{mks66}, Thm.{} 4.6]\label{mks}
Let $G=A*_H B$. Every
element of $G$ is conjugate to a cyclically reduced element of $G$.
(That is an element $g$ which
can be written as  $g = g_1\cdots g_n$ with $g_i \in A\cup B$ and either $n=1$ or $g_{i-1}$ and $g_i$ do not lie
in the same factor for all $i \in \ZZ/ n \ZZ$.) If $g$ is a cyclically reduced element of $G$ then the
following hold.
\be
\item\label{mksi} If $g$ is conjugate to $h\in H$ then $g\in A\cup B$ and there exists
a sequence $h,h_1,\ldots ,h_\ell,g$ where $h_i\in H$ and consecutive terms
are conjugate in some factor.
\item\label{mksii} If \ref{mksi} does not hold and $g$ is conjugate to an element $f\in A\cup B$, then $g$ and $f$ belong to the same factor,
$A$ or $B$, and they are conjugate in that factor.
\item\label{mksiii} If $n =|g|\geq 2$, then \ref{mksi} and \ref{mksii} do not hold. If $g$ is conjugate to a cyclically reduced element $f$, then $f$ can be written as $f=h^{-1}g_i\cdots
g_ng_1\cdots g_{i-1}h$, for some $h\in H$ and $i$ with $1\le i\le n$.
\ee
\end{theorem}
\begin{proof}
Assertion \ref{mksiii} is a trivial consequence of \refthm{thm:conjugation}.
Indeed, for $n \geq 2$ \refthm{thm:conjugation}  says  that $f=b^{-1}g_i\cdots
g_ng_1\cdots g_{i-1}b$ where $b, b^{-1}g_i,g_{i-1}b \in A\cup B$. However, 
$P$ satisfies (P6), hence $b \in H$ by \refrem{rem:p6}. Moreover, \ref{mksi} or \ref{mksii} implies  $n=1$ by \refcor{cor:hugo}, \ref{hugoi}. 

Thus, let  $n=1$ and $g ,p \in A\cup B$ be conjugate to each other. 
  Applying \refcor{cor:hugo}, \ref{hugoiii}, 
 there is a sequence $p= p_0,p_1,\ldots ,p_\ell=g$ where consecutive terms
are preconjugate, i.e., consecutive terms
are conjugate  in some factor. 
From Lemma \ref{lem:p7}, either every $p_i$ is in $H$ or the sequence may 
be shortened. Thus, if 
 $g$ is not conjugate to any $h\in H$,  we may assume $g \in A \sm H$ and
$l=1$, so
$p$ is a preconjugate of $g$; that is of the form  $\oi a g a$ for some $a\in A$.
Hence $p=\oi a g a\in A \sm H$, giving \ref{mksii}. 

Otherwise every $p_i$ is in $H$ and \ref{mksi} holds.

\end{proof}

\subsection{Conjugacy in HNN-extensions}\label{sec:hnn}
As in Section \ref{hnnag}, 
for $G = \HNN(H, t;\;  t^{-1}At =B)$ the defining pregroup can be chosen as 
$P = H \cup Ht^{-1}H \cup HtH$; and  
the base group $H$ is then equal to the canonical subgroup $G_P$. 
Therefore $P$ satisfies (P8). 

The word problem in $G$ can be solved, if we can effectively perform  Britton reductions, see e.g. in \cite{ls77}: we read non-trivial elements in $G$ as words 
over $H \sm \oneset{1}$ and in $t^{\pm 1}$.   
Whenever we see a factor in $t^{-1}At$, then  we replace it by the corresponding factor 
in $B$. Similarly, whenever we see a factor in $tBt^{-1}$, then  we replace it by the corresponding factor 
in $A$. This leads to a normal form where each $g$ becomes an element in $H$:
that is, for some uniquely defined $t$-sequence of minimal length $n\geq 1$, the element $g$  has the form 
\[g = h_0t^{\eps_1}h_1 \cdots t^{\eps_{n-1}}h_{n-1}t^{\eps_n}h_{n}.\]

In order to perform a  \emph{cyclic reduction} we remove  $h_0t^{\eps_1}h_1$
from the left and put it at the right. We continue with  Britton and cyclic reductions for as long as possible and eventually reach 
a (Britton) cyclically reduced form. Clearly every cyclically reduced 
form $g$, with non-trivial $t$-sequence, is conjugate to and element of the
form 
\begin{equation}\label{eq:crhnn}
t^{\eps_1}z_1 \cdots t^{\eps_{n-1}}z_{n-1}t^{\eps_n}z_{n},
\end{equation}
where $t^{\eps_1}, \ldots, t^{\eps_n}$ is the $t$-sequence of $g$, $z_i\in H$
and, for all $i$ we have 
\[t^{\eps_i}z_i t^{\eps_{i+1}}\notin t^{-1}At\cup tBt^{-1},\]
(subscripts modulo $n$).   Cyclically reduced elements which either belong
to $H$ or are written in the form
of \eqref{eq:crhnn} are called \emph{standard} cyclically reduced elements
of $G$. In terms of the pregroup $P$, every pregroup cyclically reduced word
can be written as a preconjugate, by an element of $H$, 
 of a standard cyclically reduced word; 
and conversely, every standard cyclically reduced
word is cyclically reduced with respect to $P$. 

The conjugacy theorem for HNN-extensions, Collins' Lemma, can be found in  \cite[Chapter IV, Theorem 2.5]{ls77}, and is stated for
standard cyclically reduced words. 
In analogy to amalgamated products  we restate it as follows. 

\begin{theorem}[D.J. Collins (1969)]
\label{th:Collins}  Let $G = \HNN(H, t;\;  t^{-1}At =B)$ 
be an HNN-extension over 
two isomorphic subgroups $A,B$ in some base group $H$. 
Every element of $G$ is conjugate to a standard cyclically reduced element.
  Let $g$ and $f$ be
conjugate, 
standard cyclically reduced elements of $G$.  
Then the following (mutually exclusive) statements hold. 
\begin{enumerate}
\item\label{coli} If $f\in A\cup B$ then there exists a
 sequence $f= c_0, c_1,\dots, c_\ell=g$ of elements of $A\cup B$, such
that, for $i=1,\ldots, \ell$,  we have 
$c_i=k_i^{-1}t^{-\del_i}c_{i-1}t^{\del_i}k_i$, with  
$k_i\in H$, $\del_i=\pm 1$  and $t^{-\del_i}c_{i-1}t^{\del_i}\in t^{-1}At\cup tBt^{-1}$.

\item\label{colii} If $g$ is not conjugate to an element of $A\cup B$ and 
$f\in H$  then $g$ and $f$ are conjugate by an element of
 the base group $H$.

\item\label{coliii} If $f$ is not  in $H$ then there exist 
 $n\ge 1$, $z_i\in H$, $\eps_i=\pm 1$, $1\le j\le n$ 
and $c\in A\cup B$, such that  
$g=t^{\eps_1}z_1\cdots t^{\eps_n}z_n$  and  
 $f$ has $t$-sequence of length $n$ and is equal in $G$ to  
\[c^{-1}t^{\eps_j}z_j\cdots
t^{\eps_n}z_nt^{\eps_1}z_1\cdots t^{\eps_{j-1}}z_{j-1}c,\]
with $c \in A$, if $\eps_j = -1$; and $c \in B$, if $\eps_j = 1$.
\end{enumerate}
\end{theorem}

\begin{proof}
As in 
the proof of \refthm{mks}, it follows from \refthm{thm:conjugation} that
 if $g=t^{\eps_1}z_1\cdots t^{\eps_n}z_n$,
 where $n\ge 2$,  then $f$  is equal in $G$ to
 $c^{-1}t^{\eps_j}z_j\cdots
t^{\eps_n}z_nt^{\eps_1}z_1\cdots t^{\eps_{j-1}}z_{j-1}c$, for some  $c\in P$, 
and as (P6) holds we have $c\in H$.

The pregroup $P$ for $G$ satisfies (P8) and $H$ is the canonical subgroup. 
Therefore $P\sm H= HtH \cup Ht^{-1}H$. 
Hence, if $f=t^{\eps}z$, for some $z\in H$ and $\eps = \pm 1$, then from
\refcor{cor:hugo} and \reflem{lem:p8} every cyclically reduced conjugate of $f$ has the 
form  
$c^{-1}t^{\eps}z c$, for some $c\in H$.  Since $g$ and $f$ are standard, 
statement
\ref{coliii} holds in both these cases.

This leaves the case where $f\in H$. As in the  proof of \refthm{mks}, 
applying \refcor{cor:hugo}, \ref{hugoiii}, 
 there is a sequence $f= p_0,p_1,\ldots ,p_\ell=g$ of elements of $P$, 
where consecutive terms
are preconjugate, say  $p_i=q_i^{-1}p_{i-1}q_i$, with $q_i\in P$. 
If $p_i\in P\sm H$ then, from \reflem{lem:p8}, the
sequence may be shortened, at least while $\ell>1$. Then, since $p_{\ell-1}\in H$ we have, from (P8), $g\in H$. Hence we may assume that either $p_i\in H$, for 
$i=0,\ldots, \ell$.     
 Again, if $q_i\in H$ and $\ell>1$ then the sequence may be shortened,  
so we may assume that 
either 
$q_i\in P\sm H$, for $i=1,\ldots , \ell$; or that $\ell=1$. 

Applying (P8), $q_i^{-1}p_{i-1}q_i=p_i\in H$, with $q_i\notin H$, implies
that $p_i$ is conjugate to an element of $A\cup B$.  
If $g$ is not conjugate to an element of $A\cup B$ it follows that $\ell=1$, 
and $q_1\in H$, as required in statement
\ref{colii}. Otherwise \ref{coli} holds.  
\end{proof}

\section{The conjugacy problem in  virtually free groups}\label{sec:cvfg}

We consider only the case of finitely generated virtually free groups. 
Virtually free groups are hyperbolic, and
it has been shown by Epstein and Holt \cite{EpsteinH06} that the 
conjugacy problem for hyperbolic groups 
can be solved in linear time. Hence,   the following is a special case of \cite{EpsteinH06}. However, our algorithm is much simpler and 
more direct. 
It can be implemented in a straightforward way using finite pregroups. 

\begin{proposition}\label{prop:lintime}
The conjugacy problem in finitely generated  virtually free groups can be solved in linear time.  
\end{proposition}

\begin{proof}
 A  finitely generated  virtually free group $G$ is  the universal 
 group  $U(P)$ of some finite pregroup $P$, see  \refprop{prop:grim}. 
 As above let $\Gam = P \sm \{\eps\}$. 
 
 By a standard procedure involving \refthm{thm:stall} we can compute
 cyclically reduced elements  in linear time. Thus, we may assume that our input 
 words are given as $g = g_1 \cdots g_n$ and $f = f_1 \cdots f_n$ with $g_i$, $f_i \in \Gam$ such that both sequences are cyclically reduced. 
 For $n=1$ we can use table look-up. Hence we may assume $n\geq2$ henceforth. 
 
 Now, let us put a linear order on $\Gam$. Then the shortlex normal form of $g$
 begins with a letter $[g_1\ov {a_1}]$ such that the geodesic length 
 of $a_1g_2 \cdots g_n$ is $n-1$. But this implies $a_1g_2 \in P$. 
 Thus, working from left to right we may compute   the shortlex normal form of $g$, in linear time; and we may assume that this is $g_1 \cdots g_n$

 We know $f = [bg_{i}] \cdots g_n g_1\cdots [g_{i-1}\ov b] \in U(P)$ by 
 \refthm{thm:conjugation}. Thus, $\ov b f b = g_{i} \cdots g_n g_1\cdots g_{i-1}$ and the number of all $\ov b f b$ is bounded by a constant dependent
only on the order of $P$. Hence, 
 we may assume that $f = g_{i} \cdots g_n g_1\cdots g_{i-1}$. Since the 
 word problem in  finitely generated  virtually free groups can be solved in linear time
 (e.g., using the system $S(P)$ or by computing the \slnf), we 
 may assume $2 < i < n$. (We also 
 see that the conjugacy problem  can be solved in quadratic time: but our 
 goal is linear time.) 
 
 Now, the  shortlex normal form of $g^2$ can be written as
 \[g_1 \cdots g_{n-1}[g_n\ov {a_1}][a_1g_1\ov {a_2}]\cdots [a_{n}g_n],\]
for appropriate $a_i\in P$. 
 As a consequence, 
 \[f\ov a_{i} = g_{i} \cdots g_{n-1}[g_n\ov {a_1}][a_1g_1\ov {a_2}]\cdots [a_{i-1}g_{i-1}\ov a_{i}].\] 
 However, the word $g_{i} \cdots g_{n-1}[g_n\ov {a_1}][a_1g_1\ov {a_2}]\cdots [a_{i-1}g_{i-1}\ov a_{i}]$ is in  shortlex normal form. Therefore the shortlex normal form
 of $f$ is $f'[a_{i-1}g_{i-1}]$, where \[f'= g_{i}\cdots g_{n-1} [g_n\ov {a_1}][a_1g_1\ov{a_2}]\cdots 
 [a_{i-2}g_{i-2}\ov a_{i-1}].\]

 Thus, it is enough to compute the shortlex normal form $\wh {f} = f'p$,
of $f$.  
 Erasing the last letter $p$ yields $f'$.  We can run 
 the  pattern matching algorithm of Knuth-Morris-Pratt,   in linear time, 
 in order to 
 obtain a list $(i_1, \ldots, i_k)$ with $2 < i_j < n$ where the 
 pattern $f'$ appears as $g_{i_j} \cdots g_{n-1} [g_n\ov {a_1}][a_1g_1\ov {a_2}]\cdots [a_{i_j-2}g_{i_j-2}\ov a_{i_j-1}].$
 All that  remains is to verify  whether or  not $[p \ov a_{i_j}]= [a_{i_j-1}g_{i_j-1}\ov a_{i_j}]$, for one index in the list.

 \end{proof}
\bibliographystyle{abbrv}

\begin{thebibliography}{10}

\bibitem{bogopol08}
O.~Bogopolski.
\newblock {\em Introduction to group theory}.
\newblock European Mathematical Society, 2008.

\bibitem{bo93springer}
R.~Book and F.~Otto.
\newblock {\em String-Rewriting Systems}.
\newblock Springer-Verlag, 1993.

\bibitem{BMR2}
A.~V. {Borovik}, A.~G. {Myasnikov}, and V.~N. {Remeslennikov}.
\newblock Algorithmic stratification of the conjugacy problem in miller{'}s
  groups.
\newblock {\em Int. J. Algebra. Comput.}, 17(5 \& 6):963--997, 2007.

\bibitem{BMR3}
A.~V. {Borovik}, A.~G. {Myasnikov}, and V.~N. {Remeslennikov}.
\newblock The conjugacy problem in amalgamated products {I}: regular elements
  and black holes.
\newblock {\em Int. J. Algebra. and Comp.}, 17(7):1299--1333, 2007.

\bibitem{BMR4}
A.~V. {Borovik}, A.~G. {Myasnikov}, and V.~N. {Remeslennikov}.
\newblock The conjugacy problem in {HNN}-extensions {I}: regular elements,
  black holes and generic complexity.
\newblock {\em Vestnik OMGU}, Special Issue:103--110, 2007.

\bibitem{Chouraqui11}
F.~Chouraqui.
\newblock The knuth-bendix algorithm and the conjugacy problem in monoids.
\newblock {\em Semigroup Forum}, 82(1):181--196, 2011.

\bibitem{ddm10}
V.~Diekert, A.~J. Duncan, and A.~G. Myasnikov.
\newblock Geodesic rewriting systems and pregroups.
\newblock In O.~Bogopolski, I.~Bumagin, O.~Kharlampovich, and E.~Ventura,
  editors, {\em Combinatorial and Geometric Group Theory}, Trends in
  Mathematics, pages 55--91. Birkh\"auser, 2010.

\bibitem{EpsteinH06}
D.~Epstein and D.~Holt.
\newblock The linearity of the conjugacy problem in word-hyperbolic groups.
\newblock {\em International Journal of Algebra and Computation}, 16:287--306,
  2006.

\bibitem{FMR1}
E.~Frenkel, A.~G. Myasnikov, and V.~N. Remeslennikov.
\newblock Regular sets and counting in free groups.
\newblock In O.~Bogopolski, I.~Bumagin, O.~Kharlampovich, and E.~Ventura,
  editors, {\em Combinatorial and Geometric Group Theory}, Trends in
  Mathematics, pages 93--119. Birkh\"auser, 2010.

\bibitem{GilHHR07}
R.~H. Gilman, S.~Hermiller, D.~F. Holt, and S.~Rees.
\newblock A characterisation of virtually free groups.
\newblock {\em Arch. Math. (Basel)}, 89(4):289--295, 2007.

\bibitem{Hoare88}
A.~H.~M. Hoare.
\newblock Pregroups and length functions.
\newblock {\em Math. Proc. Cambridge Philos. Soc.}, 104(1):21--30, 1988.

\bibitem{Horadam1983}
K.~J. Horadam.
\newblock The conjugacy problem for graph products with cyclic edge groups.
\newblock {\em Proceedings of the American Mathematical Society}, 87(3):pp.
  379--385, 1983.

\bibitem{Horadam1989}
K.~J. Horadam.
\newblock The conjugacy problem for finite graph products.
\newblock {\em Proceedings of the American Mathematical Society}, 106(3):pp.
  589--592, 1989.

\bibitem{HoradamFarr1994}
K.~J. Horadam and G.~E. Farr.
\newblock The conjugacy problem for hnn extensions with infinite cyclic
  associated groups.
\newblock {\em Proceedings of the American Mathematical Society}, 120(4):pp.
  1009--1015, 1994.

\bibitem{jan88eatcs}
M.~Jantzen.
\newblock {\em Confluent String Rewriting}, volume~14 of {\em EATCS Monographs
  on Theoretical Computer Science}.
\newblock Springer-Verlag, 1988.

\bibitem{Karrass}
A.~Karrass, A.~Pietrowski, and D.~Solitar.
\newblock Finite and infinite cyclic extensions of free groups.
\newblock {\em Journal of the Australian Mathematical Society},
  16(04):458--466, 1973.

\bibitem{KMP}
D.~{Knuth}, J.~H. {Morris}, and V.~{Pratt}.
\newblock Fast pattern matching in strings.
\newblock {\em SIAM J. Comput.}, 6:323--350, 1977.

\bibitem{Lockhart1989}
J.~M. Lockhart.
\newblock An hnn-extension with cyclic associated subgroups and with unsolvable
  conjugacy problem.
\newblock {\em Transactions of the American Mathematical Society}, 313(1):pp.
  331--345, 1989.

\bibitem{Lockhart1993}
J.~M. Lockhart.
\newblock The conjugacy problem for graph products with finite cyclic edge
  groups.
\newblock {\em Proceedings of the American Mathematical Society}, 117(4):pp.
  897--898, 1993.

\bibitem{ls77}
R.~E. Lyndon and P.~E. Schupp.
\newblock {\em Combinatorial group theory}.
\newblock Springer-Verlag, Heidelberg, 1977.

\bibitem{mks66}
W.~Magnus, A.~Karrass, and D.~Solitar.
\newblock {\em Combinatorial Group Theory}.
\newblock Interscience Publishers (New York), 1966.
\newblock Reprint of the 2nd edition (1976): 2004.

\bibitem{mat73}
{\Yu}.~Matiyasevich.
\newblock Real-time recognition of the inclusion relation.
\newblock {\em Journal of Soviet Mathematics}, 1:64--70, 1973.
\newblock Translated from Zapiski Nauchnykh Seminarov Leningradskogo Otdeleniya
  Matematicheskogo Instituta im.\ V.~A.~Steklova Akademii Nauk SSSR, Vol.~20,
  pp.~104--114, 1971.

\bibitem{ms83}
D.~E. Muller and P.~E. Schupp.
\newblock Groups, the theory of ends, and context-free languages.
\newblock {\em Journal of Computer and System Sciences}, 26:295--310, 1983.

\bibitem{Rimlinger87a}
F.~Rimlinger.
\newblock Pregroups and {B}ass-{S}erre theory.
\newblock {\em Mem. Amer. Math. Soc.}, 65(361):viii+73, 1987.

\bibitem{Rozenberg97}
G.~Rozenberg, editor.
\newblock {\em Handbook of graph grammars and computing by graph
  transformation: volume I. foundations}.
\newblock World Scientific Publishing Co., Inc., River Edge, NJ, USA, 1997.

\bibitem{serre80}
J.-P. {Serre}.
\newblock {\em Trees}.
\newblock Springer, 1980.

\bibitem{Stallings71}
J.~R. Stallings.
\newblock {\em Group theory and three-dimensional manifolds}.
\newblock Yale University Press, New Haven, Conn., 1971.
\newblock A James K. Whittemore Lecture in Mathematics given at Yale
  University, 1969, Yale Mathematical Monographs, 4.

\bibitem{Stallings87}
J.~R. Stallings.
\newblock Adian groups and pregroups.
\newblock In {\em Essays in group theory}, volume~8 of {\em Math. Sci. Res.
  Inst. Publ.}, pages 321--342. Springer, New York, 1987.

\end{thebibliography}
\newcommand{\Ju}{Ju}\newcommand{\Ph}{Ph}\newcommand{\Th}{Th}\newcommand{\Ch}{C%
h}\newcommand{\Yu}{Yu}\newcommand{\Zh}{Zh}

\end{document}